\documentclass[a4paper,11pt]{article}

\usepackage{amsfonts,qtree,stmaryrd,textcomp}
\usepackage{pifont,amssymb,amsmath,amsthm,tabularx,graphicx}
\usepackage{tree-dvips,pictexwd,dcpic}
\usepackage{bbm}
\usepackage{bm}

\usepackage{stmaryrd}

\usepackage[T1]{fontenc}
\usepackage[latin1]{inputenc}
\usepackage[text={17cm,26cm},centering]{geometry}

\usepackage{etex}
\usepackage{fancyhdr}
\usepackage{graphicx}
\usepackage{enumitem}
\usepackage{amsmath}
\usepackage[bbgreekl]{mathbbol}
\usepackage{amssymb}
\usepackage{amsthm}
\usepackage[all]{xy}

\usepackage{epigraph}
\RequirePackage{bussproofs}

\usepackage{mathabx}

\usepackage{scalerel,stackengine}

\usepackage{framed}
\usepackage{mathtools}
\usepackage{url}
\usepackage{proof}
\usepackage{xspace}
\usepackage{listings}
\usepackage{wrapfig}
\usepackage{tikz}
\usepackage{tikz-cd}

\usepackage{relsize}

\usetikzlibrary{positioning}
\usetikzlibrary{shapes}

\usetikzlibrary{arrows}
\usetikzlibrary{calc}
\usepackage{tikz-3dplot}
\usetikzlibrary{decorations.pathmorphing}

\usepackage{amsfonts,amssymb,amsmath,qtree,stmaryrd,textcomp, amsthm}

\usetikzlibrary{decorations.pathreplacing}
\tikzset{axis/.style={&lt;-&gt;}}

\stackMath
\newcommand\reallywidehat[1]{%
\savestack{\tmpbox}{\stretchto{%
  \scaleto{%
    \scalerel*[\widthof{\ensuremath{#1}}]{\kern-.6pt\bigwedge\kern-.6pt}%
    {\rule[-\textheight/2]{1ex}{\textheight}}
  }{\textheight}%
}{0.5ex}}%
\stackon[1pt]{#1}{\tmpbox}%
}
 \definecolor{MyBlue}{rgb}{0.05, 0.25, 0.65}
 
 \definecolor{MyRed}{rgb}{0.90, 0.05, 0.05}
 
\definecolor{MyGreen}{rgb}{0.05, 0.90, 0.05}

\newcommand{\B}{\boldsymbol}
\newcommand{\C}[1]{\mathcal{#1}}
\newcommand{\D}[1]{\mathbb{#1}}

\usepackage{ mathrsfs }

\newtheorem{theorem}{Theorem}[section]
\newtheorem{proposition}[theorem]{Proposition}
\newtheorem{lemma}[theorem]{Lemma}

\newtheorem{remark}[theorem]{Remark}

\newtheorem{definition}[theorem]{Definition}

\newcommand{\Nat}{{\mathbb N}}
\newcommand{\Real}{{\mathbb R}}

\newcommand{\BR}{{\Bic(\Real)}}
\newcommand{\BX}{{\Bic(X)}}

\newcommand{\id}{\mathrm{id}}

\newcommand{\Const}{\mathrm{Const}}

\newcommand{\Ap}{\mathrm{Ap}}

\newcommand{\Bic}{\mathrm{Bic}}

\newcommand{\ev}{\mathrm{ev}}

\newcommand{\Mor}{\mathrm{Mor}}

\newcommand{\Con}{\mathrm{Con}}

\newcommand{\Top}{\mathrm{\mathbf{Top}}}

\newcommand{\Cat}{\mathrm{\mathbf{Cat}}}

\newcommand{\Id}{\mathrm{Id}}
\newcommand{\BS}{\mathrm{BS}}

\newcommand{\Bis}{\mathrm{\mathbf{Bis}}}

\newcommand{\TOT}{\Leftrightarrow}

\newcommand{\To}{\Rightarrow}

\newcommand{\Too}{\Longrightarrow}

\newcommand{\sto}{\rightsquigarrow}

\newcommand{\pr}{\textnormal{\texttt{pr}}}

\newcommand{\BST}{\mathrm{BST}}

\newcommand{\Disj}{\B \rrbracket \B \llbracket}

\newcommand{\Set}{\mathrm{\mathbf{Set}}}
\newcommand{\Setc}{\Set^{\#}}

\newcommand{\Pred}{\mathrm{\mathbf{Pred}}}
\newcommand{\Predc}{\Pred^{\#}}

\newcommand{\se}{\mathrm{se}}

\newcommand{\eto}{\hookrightarrow}

\newcommand{\Dm}{\textnormal{\texttt{Dom}}}

\newcommand{\Chu}{\textnormal{\textbf{Chu}}}

\newcommand{\CSMC}{\mathrm{csms}}

\newcommand{\op}{\mathrm{op}}

\newcommand{\Fun}{\textnormal{\textbf{Fun}}}

\newcommand{\cnt}{\mathrm{cnt}}

\newcommand{\Inf}{\textnormal{\textbf{Inf}}}

\newcommand{\crTop}{\textnormal{\textbf{crTop}}}

\newcommand{\fin}{\textnormal{\texttt{fin}}}

\newcommand{\Aff}{\textnormal{\textbf{Aff}}}

\newcommand{\Hom}{\mathrm{Hom}}

\newcommand{\Sub}{\textnormal{\textbf{Sub}}}

\newcommand{\Groth}{\textnormal{\textbf{Groth}}}

\newcommand{\ccCat}{\textnormal{\textbf{ccCat}}}

\newcommand{\Ob}{\mathrm{Ob}}

\newcommand{\Rel}{\textnormal{\textbf{Rel}}}

\newcommand{\CCC}{\mathrm{ccc}}

\newcommand{\End}{\mathrm{End}}

\newcommand{\CHU}{\textnormal{\textbf{CHU}}}

\newcommand{\CAT}{\textnormal{\textbf{CAT}}}

\begin{document}

\date{}

\title{\textbf{Chu representations of categories related to constructive mathematics}}

\author{Iosif Petrakis\\	
Mathematics Institute, Ludwig-Maximilians-Universit\"{a}t M\"{u}nchen\\
petrakis@math.lmu.de}  

%





\maketitle

\begin{abstract}
\noindent 
If $\C C$ is a closed symmetric monoidal category, the Chu category $\Chu(\C C, \gamma)$ over 
$\C C$ and an object $\gamma$ of it was defined by Chu in~\cite{Ba79}, as a $*$-autonomous category generated from $\C C$. In~\cite{Bi67} Bishop introduced the category $\B {\C P}^{\Disj}(X)$ of complemented subsets of a set $X$, in order to overcome 
the problems generated by the use of negation in constructive measure theory. In~\cite{Sh18} Shulman mentions
that Bishop's complemented subsets correspond roughly to the Chu construction. In this paper we explain
this correspondence by showing that there is a Chu representation (a full embedding) of $\B {\C P}^{\Disj}(X)$ into 
$\Chu(\Set, X \times X)$. A Chu representation of the category of Bishop spaces into $\Chu(\Set, \Real)$ is shown, as 
the constructive analogue to the standard Chu representation of the category of topological spaces into $\Chu(\Set, \D 2)$.
In order to represent the category of predicates (with objects pairs $(X, A)$, where $A$ is a subset of $X$, and the category
of complemented predicates (with objects pairs $(X, A)$, where $A$ is a complemented subset of $X$, we generalise the 
Chu construction by defining the Chu category over a cartesian closed category $\C C$ and an endofunctor on $\C C$.
Finally, we introduce the antiparallel Grothendieck construction over a product category and a contravariant $\Set$-valued functor on it of which the Chu construction is a special case, in case $\C C$ is a locally small, cartesian closed category.\\[2mm]
\end{abstract}


\section{Introduction}
\label{sec: intro}

In category theory the Chu construction is a method of generating a $*$-autonomous category from a closed symmetric 
monoidal category (CSMC) (see~\cite{Ri16} and~\cite{Ma17}). The $*$-autonomous categories provide models for classical (multiplicative) linear logic (in~\cite{Se89}). The Chu construction was introduced by Chu in his Master's Thesis,
and appeared first in~\cite{Ba79}. The category $\Chu(\Set, X)$ was introduced by Lafont and Streicher in~\cite{LS91} under the name of games (semantics for  linear logic). In a series of papers, see e.g.,~\cite{Pr99}, Pratt and his collaborators applied the Chu construction on topics of  theoretical computer science (e.g., concurrency). The Chu construction has been applied to hardware verification, game theory, fuzzy systems, and the foundations of quantum mechanics (see~\cite{Ab12} and~\cite{Ab18}). 
There is a rich representation theory connected to the Chu construction, as many important, and quite different categories are
 represented (fully embedded) into some Chu category. The representation of categories related to constructive mathematics into some Chu category is a major theme of this paper.
 
In~\cite{Sh18}, p.~8, Shulman mentions that ``a notion corresponding
roughly to the Chu construction was already introduced by (Bishop and Bridges in)~\cite{BB85}
under the name complemented subset''. Actually, the notion of a complemented subset is already introduced by 
Bishop in~\cite{Bi67}, pp.~66-69, under the name \textit{complemented set}. 
Here we explain in what sense Bishop's notion corresponds to the Chu construction. Namely, there is a Chu representation of 
the thin category $\C P^{\Disj}(X)$ of complemented subsets of a set $X$ into $\Chu(\Set, X \times X)$. Notice that Bishop's motivation for introducing complemented (sub)sets is rooted
to his need to overcome problems generated by the use of negation in basic set and measure theory in a constructive setting (see~\cite{Pe20}, chapter 7, and~\cite{Pe19a}). Hence, the connection described here between the Chu construction and Bishop's notion of complemented subsets seems to be accidental.

All notions and results presented here concern cartesian closed categories $(\CCC)$, although they can be generalised to symmetric monoidal closed categories\footnote{A cartesian closed category $\C C$ is a $\CSMC$ where its tensor product of $\C C$ is its product and the tensor-unit is the terminal object of $\C C$.
The category $\Rel(\Set)$ with objects sets and morhisms relations $R \subseteq X \times Y$ is a
$\CSMC$ that is not a $\CCC$.}.
%
We structure this paper as follows:

\begin{itemize}
\item In section~\ref{sec: chu1} we present the basic of the Chu construction and the local Chu functor. 
\item In section~\ref{sec: globalchu} we present the global Chu functor that corresponds to the Chu construction.
\item In section~\ref{sec: boolechurepr} we present the standard and classical boolean Chu representation of $\Top$ and the induced boolean Chu representation of the category of information systems $\Inf$.
\item In section~\ref{sec: normalchurepr} we present the constructive normal Chu representation of the category of Bishop spaces $\Bis$. This representation of $\Bis$ is the constructive analogue of the aforementioned Chu representation of $\Top$. The notion of a Bishop space is Bishop's constructive, function-theoretic alternative to the classical, set-based notion of a topological space (see~\cite{Pe15}-\cite{Pe19b} and~\cite{Pe20b}-\cite{Pe21e}).
\item In section~\ref{sec: chusubsets} and~\ref{sec: chucompl} we give the Chu representation of the category $\B {\C P}(X)$ of subsets of a set $X$ and of the category $\B {\C P}^{\Disj}(X)$ of complemented subsets of $X$, where $X$ is a set equipped with an equality $=_X$ and an inequality $\neq_X$, respectively. All set-theoretic notions mentioned here are within our reconstruction $\BST$ of Bishop's set theory found in~\cite{Bi67} and~\cite{BB85} (see~\cite{Pe19d} and, especially,~\cite{Pe20}).
\item In section~\ref{sec: chu2} we introduce the generalised Chu category over a $\CCC$ $\C C$ and an endofunctor $\Gamma$ on $\C C$.
\item In section~\ref{sec: genchufunctor} we define the generalised global Chu functor that corresponds to the generalised Chu construction.
\item With the help of the generalised Chu construction we provide a generalised Chu representation of the categories of predicates $\Pred$ and of complemented predicates $\Pred^{\neq}$ in sections~\ref{sec: predchu} and~\ref{sec: complpredchu}, respectively.
\item In section~\ref{sec: chugroth} we introduce the antiparallel Grothendieck construction over a product category and a contravariant $\Set$-valued functor on it, which has the Chu construction as a special case, in case $\C C$ is a $\CCC$.
\end{itemize}

For all notions and results from category theory that are used here without explanation or proof we refer to~\cite{MM92},~\cite{Aw10} and~\cite{Ri16}.

\section{The Chu construction over a $\CCC$ $\C C$}
\label{sec: chu1}

Unless otherwise stated, throughout this paper $\C C, \C D, \C E$ are $\CCC$ and $\gamma \in C_0, \delta \in D_0$ are object of $\C C$ and $\C D$, respectively. To show that the Chu construction in Definition~\ref{def: chu} is category, one uses the fact the product 
$\times \colon \C C \times \C C \to \C C$ 
is a bifunctor (i.e., a functor). Moreover, if 
$f \colon a \to a{'}$ and $g \colon b \to b{'}$ in $C_1$, then $f \times g \colon a \times b \to a{'} \times b{'}$,
such that $1_a \times 1_b = 1_{a \times b}$, and if $f{'} \colon a{'} \to a{''}$ and $g{'} \colon b{'} \to b{''}$ in $C_1$, then 
\begin{equation}\label{eq: cmc1}
(f{'} \times g{'}) \circ (f \times g) = (f{'} \circ f) \times (g{'} \circ g).
\end{equation}
%
%
%
If $a{'} = a{''} = a$ and $f{'} = f = 1_a$, by equation~(\ref{eq: cmc1}) we get
\begin{equation}\label{eq: cmc2}
(1_a \times g{'}) \circ (1_a \times g) = (1_a \circ 1_a) \times (g{'} \circ g) = 1_a \times (g{'} \circ g).
\end{equation}
%
%
%
Similarly, if $b{'} = b{''} = b$ and $g{'} = g = 1_b$, by equation~(\ref{eq: cmc1}) we get
\begin{equation}\label{eq: cmc4}
(f{'} \times 1_b) \circ (f \times 1_b) =  (f{'} \circ f) \times (1_b \circ 1_b) = (f{'} \circ f) \times 1_b. 
\end{equation}
If $a, c, d, j \in C_0$, $\phi \colon a \to c$ and $\theta \colon j \to d \in C_1$, then 
\begin{equation}\label{eq: cmc3}
(\phi \times 1_d) \circ (1_a \times \theta) = (1_c \times \theta) \circ (\phi \times 1_j) 
\end{equation}
\begin{center}
\begin{tikzpicture}

\node (E) at (0,0) {$\mathsmaller{a \times j} $};
\node[right=of E] (S) {};
\node[right=of S] (T) {$\mathsmaller{a \times d}$};
\node[below=of T] (F) {$\mathsmaller{c \times d}$};
\node[below=of E] (G) {$\mathsmaller{c \times j}$};

\draw[->] (E)--(T) node [midway,above] {$\mathsmaller{1_a \times \theta}$};
\draw[->] (T)--(F) node [midway,right] {$\mathsmaller{\phi \times 1_d}$};
\draw[->] (E)--(G) node [midway,left] {$\mathsmaller{\phi \times 1_j}$};
\draw[->] (G)--(F) node [midway,below] {$\mathsmaller{1_c \times \theta}$};

\end{tikzpicture}
\end{center}
\begin{align*}
 (1_c \times \theta) \circ (\phi \times 1_j) & \stackrel{(\ref{eq: cmc1})} = (1_c \circ \phi) \times (\theta \circ 1_j)\\
 & = \phi \times \theta\\
 & = (\phi \circ 1_a) \times (1_d \circ \theta)\\
 & \stackrel{(\ref{eq: cmc1})} = (\phi \times 1_d) \circ (1_a \times \theta).
\end{align*}

%
%
%

\begin{definition}[The Chu construction over a $\CCC$ $\C C$ and some $\gamma \in C_0$]\label{def: chu}
The Chu category $\Chu(\C C, \gamma)$ 
over $\C C$ and $\gamma$ has objects Chu spaces i.e., triplets $(a, f, x)$, with $a, x \in C_0$ and 
$f \colon a \times x \to \gamma \in C_1$.
A morphism $\phi \colon (a, f, x) \to 
(b, g, y)$ in  $\Chu(\C C, \gamma)$, or a Chu transform,
is a pair $\phi = \big(\phi^+, \phi^-\big)$, where $\phi^+ \colon a \to b$ and $\phi^- 
\colon y \to x$ are in $C_1$ such that the following diagram commutes
\begin{center}
\begin{tikzpicture}

\node (E) at (0,0) {$\mathsmaller{a \times y}$};
\node[right=of E] (T) {};
\node[right=of T] (F) {$\mathsmaller{a \times x}$};
\node[below=of E] (A) {$\mathsmaller{b \times y}$};
\node[below=of F] (S) {$\mathlarger{\gamma}.$};

\draw[->] (E)--(A) node [midway,left] {$\mathsmaller{\phi^+ \times 1_y}$};
\draw[->] (F)--(S) node [midway,right] {$\mathsmaller{f}$};
\draw[->] (E)--(F) node [midway,above] {$\mathsmaller{1_a \times \phi^-}$};
\draw[->] (A)--(S) node [midway,below] {$\mathsmaller{g}$};

\end{tikzpicture}
\end{center}
If $\theta = \big(\theta^+, \theta^-\big) \colon (b, g, y) \to (c, h, z)$, then $\theta \circ \phi = 
\big(\theta^+ \circ \phi^+, \phi^- \circ \theta^-\big)$. Moreover, $1_{(a, f, x)} = (1_a, 1_x)$.

\end{definition}

%
%
%
%
%
%

If $\C C$ is bicomplete (complete and cocomplete), then $\Chu(\C C, \gamma)$ is also bicomplete (see~\cite{Ma17}, p.~41.
The following result is standard (see also~\cite{Ab12}, p.~712).

\begin{proposition}[The local Chu functor]\label{prp: intfunctor}
The rule $\Chu^{\C C} \colon \C C \to \Cat$, defined by
$$\Chu_0^{\C C}(\gamma) = \Chu(\C C, \gamma),$$
$$\Chu_1^{\C C}(u \colon \gamma \to \delta) = u_* \colon \Chu(\C C, \gamma) \to \Chu(\C C, \delta),$$
$$(u_*)_0(a, f, b) = (a, u \circ f, b),$$
\begin{center}
\begin{tikzpicture}

\node (E) at (0,0) {$a \times b$};
\node[right=of E] (T) {$\gamma$};
\node[right=of T] (F) {$\delta$};

\draw[MyBlue,right hook->] (E)--(T) node [midway,above] {$f$};
\draw[MyBlue,right hook->] (T)--(F) node [midway,above] {$u$};

\end{tikzpicture}
\end{center}
$$(u_*)_1\big(\phi^+, \phi^-\big) = \big(\phi^+, \phi^-\big),$$
is a functor.
Moreover, if $u$ is a monomorphism, then $u_*$ is a full embedding. 
\end{proposition}

Let $\Set$ be the $\CCC$ of sets and functions in Bishop's sense\footnote{One could have considered some other constructive approach to set theory, like Aczel's constructive set theory in ~\cite{AR10}. Most of the results presented here hold also for sets in a classical sense.}. If $(A, f, B)$ and $(C, g, D)$ are Chu spaces in  
$\Chu(\Set, X)$, for some given set $X$, and if $(\phi^+, \phi^-) \colon (A, f, B) \to (C, g, D)$, then  
the commutativity of the rectangle
\begin{center}
\begin{tikzpicture}

\node (E) at (0,0) {$A \times D$};
\node[right=of E] (T) {};
\node[right=of T] (F) {$A \times B$};
\node[below=of E] (A) {$C \times D$};
\node[below=of F] (S) {$X$};

\draw[->] (E)--(A) node [midway,left] {$\mathsmaller{\phi^+ \times \id_{D}}$};
\draw[->] (F)--(S) node [midway,right] {$f$};
\draw[->] (E)--(F) node [midway,above] {$\mathsmaller{\id_{A} \times \phi^-}$};
\draw[->] (A)--(S) node [midway,below] {$g$};

\end{tikzpicture}
\end{center}
is written as $f\big(a, \phi^-(d)\big) = g\big(\phi^+(a), d\big)$, for every $a \in A$ and $d \in D$.
In the next two definitions we follow~\cite{Pr99} and~\cite{GT07}, 
respectively.

\begin{definition}\label{def: extsep}
 A Chu space $(A, f, B)$ in $\Chu(\Set, X)$ is called separable, if $\widehat{f} \colon A \to (B \to X)$, where
 $$\big[\widehat{f}(a)\big](b) = f(a, b),$$
 for every $a \in A$ and $b \in B$, is an injection. A Chu space $(A, f, B)$ in $\Chu(\Set, X)$ is called extensional,
 if $\widecheck{f} \colon B \to (A \to X)$, where
 $$\big[\widecheck{f}(b)\big](a) = f(a, b),$$
 for every $b \in B$ and $a \in A$, is an injection. If $(A, f, B)$ is both separable and extensional, it is called
 biextensional. If $B \subset X^A$ and $f \colon A \times B \to X$ is defined by $f(a, b) = b(a)$, then $(A, f, B)$ 
 is called a normal Chu space. The Chu spaces in $\Chu(\Set, \D 2)$ are called Boolean.   

\end{definition}

\begin{definition}\label{def: affine}
If $\C C$ is a category and $\gamma \in C_0$, the affine category
$\Aff(\C C, \gamma)$ over $\C C$ and $\gamma$ has objects pairs $(a, F)$, where $a \in C_0$ and
$F \subseteq C_1(a, \gamma) = \Hom(a, \gamma)$, and a morphism $h \colon (a, F) \to (b, G)$ in $\Aff(\C C, \gamma)$
is a morphism 
$h \colon a \to b$ in $C_1$ such that $g \circ h \in F$, for every $g \in G$.
\end{definition}

Next we fix some basic terminology.


\begin{definition}\label{def: churepresent}
Let $\C C, \C D$ be categories and $F \colon C \to \C D$ a functor. $F$ is an embedding, if it is injective on objects and
faithful, and its is a representation, if it is a
full embedding. If $\C D$ is a Chu category and $F$ is a representation, we call $F$ a 
Chu representation. We call a Chu representation $F$ strict, if $F$ is injective on arrows. We call a Chu representation 
boolean $($normal$)$, if $F_0(a)$ is a Boolean $($normal$)$ Chu space, for every $a \in C_0$. 
 
\end{definition}

All Chu representations included in this paper are going to be strict. If $\C C$ is a $\CCC$, let $\ev_{\gamma, a} \colon a \times \gamma^a \to \gamma$ in $C_1$ such that for every $f \colon a \times b \to \gamma$ there is a unique $\widehat{f} \colon b \to \gamma^a$ with $f = \ev_{\gamma, a} \circ \big(1_a \times \widehat{f}\big)$. The next result is also standard, and its proof 
is constructive. The normal Chu representation of $\Set$ through $E^{\Set, \D 2}$
into $\Chu(\Set, \D 2)$ is classically the ``same'' to the boolean Chu representation of $\Set$ into $\Chu(\Set, \D 2)$ in section~\ref{sec: boolechurepr}, which relies though, on the classical treatment of negation.

\begin{proposition}[Chu representation of a $\CCC$]\label{prp: cccrepr}
 The functor
 $E^{\C C, \gamma} \colon \C C \to \Chu(\C C, \gamma)$, defined by
 $$E^{\C C, \gamma}_0(a) = \big(a, \ev_{\gamma, a}, \gamma^a\big), $$
 $$E^{\C C,\gamma}_1(f \colon a \to b) = (f, f^-) \colon \big(a, \ev_{\gamma, a}, \gamma^a\big) \to \big(b, \ev_{\gamma, b}, 
 \gamma^b\big),$$
 $$f^- = \widehat{h} \colon \gamma^b \to \gamma^a, \ \ \  h = \ev_{\gamma, b}  \circ  \big(f \times 1_{\gamma^b}\big),$$
 \begin{center}
\begin{tikzpicture}

\node (E) at (0,0) {$\mathsmaller{a \times \gamma^a}$};
\node[right=of E] (T) {};
\node[right=of T] (F) {$ \ \mathsmaller{\gamma}$};
\node[below=of E] (B) {};
\node[below=of B] (A) {$ \ \mathsmaller{a \times \gamma^b}$};
\node[below=of T] (S) {$ \ \mathsmaller{b \times \gamma^b} $};

\draw[->] (E)--(F) node [midway,above] {$\mathsmaller{\ev_{\gamma, a}}$};
\draw[->] (A)--(E) node [midway,left] {$\mathsmaller{1_a \times \widehat{h}} \ $};
\draw[MyBlue,->] (A)--(S) node [midway,right] {$ \ \mathsmaller{f \times 1_{\gamma^b}}$};
\draw[MyBlue,->] (S)--(F) node [midway,right] {$ \ \mathsmaller{\ev_{\gamma, b}}$};

\end{tikzpicture}
\end{center}
 is a strict Chu representation of $\C C$ into $\Chu(\C C, \gamma)$.
\end{proposition}

\section{The global Chu functor}
\label{sec: globalchu}

 If a functor $F \colon \C C \to \C D$ preserves products
 (i.e., binary product diagrams), then for every $a, b \in C_0$ there is a unique morphism
 $F_{ab} \colon  F_0(a) \times F_0(b) \to F_0(a \times b)$, which is an isomorphism
 \begin{center}
\begin{tikzpicture}

\node (E) at (0,0) {$a$};
\node[right=of E] (T) {$a \times b$};
\node[right=of T] (F) {$b$};

\draw[->] (T)--(E) node [midway,above] {$\pr_a$};
\draw[->] (T)--(F) node [midway,above] {$\pr_b$};

\end{tikzpicture}
\end{center} 
  \begin{center}
\begin{tikzpicture}

\node (E) at (0,0) {$F_0(a)$};
\node[right=of E] (S) {};
\node[right=of S] (T) {$F_0(a \times b)$};
\node[right=of T] (X) {};
\node[right=of X] (F) {$F_0(b)$};
\node[below=of T] (U) {$F_0(a) \times F_0(b)$.};

\draw[->] (T)--(E) node [midway,above] {$F_1(\pr_a)$};
\draw[->] (T)--(F) node [midway,above] {$F_1(\pr_b)$};
\draw[->] (U)--(E) node [midway,left] {$\pr_{F_0(a)} \ $};
\draw[->] (U)--(F) node [midway,right] {$ \ \ \ \pr_{F_0(b)}$};
\draw[MyBlue,->] (U)--(T) node [midway,right] {$F_{ab}$};

\end{tikzpicture}
\end{center}   
For every $a, a{'}, b, b{'} \in C_0$ and every $f \colon a \to a{'}, g \colon b \to b{'}$ in $C_1$
 the following rectangle commutes
\begin{center}
\begin{tikzpicture}

\node (E) at (0,0) {$\mathsmaller{F_0(a \times b)}$};
\node[right=of E] (T) {};
\node[right=of T] (F) {$\mathsmaller{F_0(a{'} \times b{'})}$};
\node[below=of E] (A) {$\mathsmaller{F_0(a) \times F_0(b)}$};
\node[below=of F] (S) {$\mathsmaller{F_0(a{'}) \times F_0(b{'})}.$};

\draw[->] (A)--(E) node [midway,left] {$\mathsmaller{F_{ab}}$};
\draw[->] (S)--(F) node [midway,right] {$\mathsmaller{F_{a{'}b{'}}}$};
\draw[->] (E)--(F) node [midway,above] {$\mathsmaller{F_1(f \times g)}$};
\draw[->] (A)--(S) node [midway,below] {$\mathsmaller{F_1(f) \times F_1(g)}$};

\end{tikzpicture}
\end{center}
If $G \colon \C D \to \C E$ also preserves products and $(G_{cd})_{c, d \in D_0}$ are the canonical isomorphisms 
$G_{cd} \colon  G_0(c) \times G_0(d) \to G_0(c \times d)$, then $G \circ F$ also preserves products and for every
$a, b \in C_0$ we have that
$$(G \circ F)_{ab} = G_1(F_{ab}) \circ G_{F_0(a)F_0(b)}$$
 
 \begin{center}
\begin{tikzpicture}

\node (E) at (0,0) {$\mathsmaller{G_0(F_0(a)) \times G_0(F_0(b))}$};
\node[right=of E] (T) {};
\node[left=of E] (S) {};
\node[below=of T] (A) {$\mathsmaller{G_0(F_0(a) \times F_0(b))}$};
\node[below=of E] (X) {};
\node[below=of X] (U) {$\mathsmaller{G_0(F_0(a \times b))}$};
\node[right=of U] (W) {};
\node[right=of W] (B) {$\mathsmaller{G_0(F_0(b))}$};
\node[left=of U] (C) {};
\node[left=of C] (D) {$\mathsmaller{G_0(F_0(a))}$};

\draw[->,bend left=60] (E) to node [midway,right] {$ \ \mathsmaller{\pr_{G_0(F_0(a))}}$} (B);
\draw[->] (E)--(D) node [midway,left] {$\mathsmaller{\pr_{G_0(F_0(a))}} \ \ $};
\draw[->] (U)--(B) node [midway,below] {$\mathsmaller{G_1(F_1(\pr_b))}$};
\draw[->] (U)--(D) node [midway,below] {$\mathsmaller{G_1(F_1(\pr_a))}$};
\draw[MyBlue,->] (E)--(U) node [midway,left] {$\mathsmaller{(G \circ F)_{ab}}$};
\draw[MyBlue,->] (E)--(A) node [midway,right] {$\mathsmaller{G_{F_0(a)F_0(b)}}$};
\draw[MyBlue,->] (A)--(U) node [midway,right] {$ \ \mathsmaller{G_1(F_{ab})}$};

\end{tikzpicture}
\end{center} 
The canonical isomorphisms of the identity functor $\Id^{\C C}$ on $\C C$ is the 
family $(1_{a \times b})_{a, b \in C_0}$.

\begin{lemma}\label{lem: lemextfunctor}
 Let $F \colon  \C C \to \C D$ be a product-preserving functor with $(F_{ab})_{a, b \in C_0}$ the canonical isomorphisms 
 of $F$, and let
 $\phi \colon F_0(\gamma) \to \delta$ 
 in $D_1$. The rule
 $F_* \colon \Chu(\C C, \gamma) \to \Chu(\C D, \delta)$, defined by
 $$(F_*)_0(a, f, b) = \big(F_0(a), \phi \circ F_1(f) \circ F_{ab}, F_0(b)\big)$$
 \begin{center}
\begin{tikzpicture}

\node (E) at (0,0) {$F_0(a) \times F_0(b)$};
\node[right=of E] (T) {$F_0(a \times b)$};
\node[right=of T] (F) {$F_0(\gamma)$};
\node[right=of F] (D) {$\delta$};

\draw[MyBlue,->] (E)--(T) node [midway,above] {$F_{ab}$};
\draw[MyBlue,->] (T)--(F) node [midway,above] {$F_1(f)$};
\draw[MyBlue,->] (F)--(D) node [midway,above] {$\phi$};

\end{tikzpicture}
\end{center}
$$(F_*)_1\big(\phi^+, \phi^-\big) \colon \big(F_0(a), \phi \circ F_1(f) \circ F_{ab}, F_0(b)\big)
\to \big(F_0(c), \phi \circ F_1(g) \circ F_{cd}, F_0(d)\big),$$
$$(F_*)_1\big(\phi^+, \phi^-\big) = \big(F_1(\phi^+), F_1(\phi^-)\big),$$
where $\big(\phi^+, \phi^-\big) \colon (a, f, b) \to (c, g, d)\big)$, is a functor.
\end{lemma}

\begin{proof}
To show that $F_*$ is well-defined, we show that $(F_*)_0\big(\phi^+, \phi^-\big) \colon 
 \big(F_0(a), \phi \circ F_1(f) \circ F_{ab}, F_0(b)\big)
\to \big(F_0(c), \phi \circ F_1(g) \circ F_{cd}, F_0(d)\big)$ i.e., the following diagram commutes
\begin{center}
\begin{tikzpicture}

\node (E) at (0,0) {$\mathsmaller{F_0(a) \times F_0(d)}$};
\node[right=of E] (T) {};
\node[right=of T] (F) {$\mathsmaller{F_0(a) \times F_0(b)}$};
\node[below=of E] (A) {$\mathsmaller{F_0(c) \times F_0(d)}$};
\node[below=of F] (S) {$\mathsmaller{\delta}.$};

\draw[->] (E)--(A) node [midway,left] {$\mathsmaller{F_1(\phi^+) \times 1_{F_0(d)}}$};
\draw[->] (F)--(S) node [midway,right] {$\mathsmaller{\phi \circ F_1(f) \circ F_{ab}}$};
\draw[->] (E)--(F) node [midway,above] {$\mathsmaller{1_{F_0(a)} \times F_1(\phi^-)}$};
\draw[->] (A)--(S) node [midway,below] {$\mathsmaller{\phi \circ F_1(g) \circ F_{cd}}$};

\end{tikzpicture}
\end{center}
By the commutativity of the following diagrams we have that
\begin{center}
\begin{tikzpicture}

\node (E) at (0,0) {$\mathsmaller{a \times d}$};
\node[right=of E] (T) {};
\node[right=of T] (F) {$\mathsmaller{a \times b}$};
\node[below=of E] (A) {$\mathsmaller{c \times d}$};
\node[below=of F] (S) {$\mathsmaller{\gamma}$};

\draw[->] (E)--(A) node [midway,left] {$\mathsmaller{\phi^+ \times 1_d}$};
\draw[->] (F)--(S) node [midway,right] {$\mathsmaller{f}$};
\draw[->] (E)--(F) node [midway,above] {$\mathsmaller{1_a \times \phi^-}$};
\draw[->] (A)--(S) node [midway,below] {$\mathsmaller{g}$};

\end{tikzpicture}
\end{center}
\begin{center}
\begin{tikzpicture}

\node (E) at (0,0) {$\mathsmaller{F_0(a \times d)}$};
\node[right=of E] (T) {};
\node[right=of T] (F) {$\mathsmaller{F_0(c \times d)}$};
\node[below=of E] (A) {$\mathsmaller{F_0(a) \times F_0(d)}$};
\node[below=of F] (S) {$\mathsmaller{F_0(c) \times F_0(d)}$};
\node[right=of F] (K) {$\mathsmaller{F_0(a \times d)}$};
\node[right=of K] (B) {};
\node[right=of B] (C) {$\mathsmaller{F_0(a \times b)}$};
\node[below=of K] (L) {$\mathsmaller{F_0(a) \times F_0(d)}$};
\node[below=of C] (M) {$\mathsmaller{F_0(a) \times F_0(b)}$};

\draw[->] (A)--(E) node [midway,left] {$\mathsmaller{F_{ad}}$};
\draw[->] (S)--(F) node [midway,right] {$\mathsmaller{F_{cd}}$};
\draw[->] (E)--(F) node [midway,above] {$\mathsmaller{F_1(\phi^+ \times 1_d)}$};
\draw[->] (A)--(S) node [midway,below] {$\mathsmaller{F_1(\phi^+) \times F_1(1_d)}$};
\draw[->] (L)--(K) node [midway,left] {$\mathsmaller{F_{ad}}$};
\draw[->] (M)--(C) node [midway,right] {$\mathsmaller{F_{ab}}$};
\draw[->] (K)--(C) node [midway,above] {$\mathsmaller{F_1(1_a \times \phi^-)}$};
\draw[->] (L)--(M) node [midway,below] {$\mathsmaller{F_1(1_a) \times F_1(\phi^-)}$};

\end{tikzpicture}
\end{center}
\begin{align*}
\phi \circ F_1(f) \circ F_{ab} \circ [1_{F_0(a)} \times F_1(\phi^-)] & =  \phi \circ F_1(f) \circ F_1(1_a \times \phi^-)
\circ F_{ad}\\
& = \phi \circ F_1(g) \circ F_1(\phi^+ \times 1_d) \circ F_{ad}\\
& = \phi \circ F_1(g) \circ F_{cd} \circ [F_1(\phi^+) \times F_1(1_d)]\\
& = \phi \circ F_1(g) \circ F_{cd} \circ [F_1(\phi^+) \times 1_{F_0(d)}].
\end{align*}
The preservation of the units and compositions by $F_*$ are immediate to show.
\end{proof}

If $\eta \colon F \To G$, we cannot define a natural transformation $\eta_* \colon F_* \To G_*$ i.e., 
we cannot show that $F \mapsto F_*$ is a functor on the category $\Fun^{\times}(\C C, \C D)$ of product-preserving functors from 
$\C C$ to $\C D$. What we showed though, in the previous lemma is that the \textit{pair} $(F, \phi)$ generated the 
functor $F_* \colon \Chu(\C C, \gamma) \to \Chu(\C D, \delta)$. Next we describe an instance of the (generalised)
covariant Grothendieck construction that defines the category with respect to which $(F, \phi) \mapsto F_*$
becomes a functor.

\begin{definition}[A covariant Grothendieck construction]\label{def: ccCat}
 Let $\ccCat$ be the category of cartesian closed categories with morphisms the product preserving functors\footnote{One 
 could have considered the cartesian closed functors i.e., the functors preserving the whole structure of a cartesian closed
 category, as morphisms of $\ccCat$.}. The Grothendieck category 
 $$\Groth\big(\ccCat, \Id^{\ccCat}\big)$$
 over $\ccCat$ and the 
 covariant identity functor $\Id^{\ccCat} \colon \ccCat \to \Cat$ has objects pairs $(\C C, \gamma)$, where 
 $\C C$ is a cartesian closed category and 
 $\gamma \in \Ob_{ \Id^{\ccCat}_0(\C C)} = C_0$. A morphism $(F, \phi) \colon (\C C, \gamma) \to (\C D, \delta)$ is a
 product-preserving functor $F \colon \C C \to \C D$ and a morphism $\phi \colon  \big[\Id^{\ccCat}_1(F)\big]_0(\gamma) 
 \to \delta$ i.e., $\phi \colon F_0(\gamma) \to \delta$. If $(G, \theta) \colon (\C D, \delta) \to (\C E, \varepsilon)$, then 
 $(G, \theta) \circ (F, \phi) = \big(G \circ F, \theta \circ G_1(\phi)\big)$. Moreover, $1_{(\C C, \gamma)} = 
 \big(\Id^{\C C}, 1_{\gamma}\big)$. 
\end{definition}

\begin{theorem}[The global Chu functor]\label{prp: extfunctor}
The rule 
$\Chu \colon \Groth\big(\ccCat, \Id^{\ccCat}\big) \to \Cat$,
defined by
$$\Chu_0(\C C, \gamma) = \Chu(\C C, \gamma),$$
$$\Chu_1\big(F, \phi) \colon (\C C, \gamma) \to (\C D, \delta)\big) \colon \Chu(\C C, \gamma) \to \Chu(\C D, \delta),$$
$$\Chu_1\big(F, \phi) = F_*,$$
where $F_*$ is defined in Lemma~\ref{lem: lemextfunctor}, is a functor. Moreover, 
if $F \colon \C C \to \C D$ is a full embedding and
$\phi$ is a monomorphism, then $F_*$ is a full embedding of $\Chu(\C C, \gamma)$ into $\Chu(\C D, \delta)$. 
\end{theorem}

\begin{proof}
 By Lemma~\ref{lem: lemextfunctor} $\Chu_1(F, \phi)$ is well-defined. Clearly,
 $$\Chu_1(1_{(\C C, \gamma)}) = \Chu_1 \big(\Id^{\C C}, 1_{\gamma}\big) = \big[\Id^{\C C}\big]_* = 1_{\Chu(\C C, \gamma)}.$$
 If $(G, \theta) \colon (\C D, \delta) \to (\C E, \varepsilon)$, we show that 
 $(G \circ F)_* = G_* \circ F_*$. By definition $(G, \theta) \circ (F, \phi) = \big(G \circ F, \theta \circ G_1(\phi)\big)$,
 and by the equality shown for the canonical isomorphisms $[(G \circ F)_{ab}]_{a, b \in C_0}$ we get 
 \begin{align*}
  \big[(G \circ F)_*\big]_0(a, f, b) & = \big(G_0(F_0(a)), \theta \circ G_1(\phi) \circ G_1(F_1(f)) \circ (G \circ F)_{ab}, 
  G_0(F_0(b))\big)\\
  & = \big(G_0(F_0(a)), \theta \circ G_1(\phi) \circ G_1(F_1(f)) \circ G_1(F_{ab}) \circ G_{F_0(a)F_0(b)}, 
  G_0(F_0(b))\big)\\
  & = \big(G_0(F_0(a)), \theta \circ G_1\big[\phi \circ F_1(f) \circ F_{ab}\big] \circ G_{F_0(a)F_0(b)}, 
  G_0(F_0(b))\big)\\
  & = (G_*)_0\big(F_0(a), \phi \circ F_1(f) \circ F_{ab}, F_0(b)\big)\\
  & = (G_*)_0\big((F_*)_0(a, f, b)\big).
 \end{align*}
The equality $[(G \circ F)_*]_1(\phi^+, \phi^-) = (G_*)_1\big((F_*)_1(\phi^+, \phi^-)\big)$ follows immediately.
Let $F \colon \C C \to \C D$ be a full embedding and
$\phi$ a monomorphism. The equality $\big(F_0(a), \phi \circ F_1(f) \circ F_{ab}, F_0(b)\big) =
 \big(F_0(a{'}), \phi \circ F_1(f{'}) \circ F_{a{'}b{'}}, F_0(b{'})\big)$ implies $a = a{'}, b = b{'}$, and as 
 $\phi$ is a monomorphism and $F_ab$ an isomorphism, hence an epimorphism, we get $ F_1(f) =  F_1(f{'})$, hence
 $f = f{'}$. The fact that $F_*$ is faithful and full follows immediately.
\end{proof}

The local Chu functor is a special case of the global one. Namely,
$$\Chu_1(\Id^{\C C}, u \colon \gamma \to \delta) = u_* = \Chu^{\C C}_1(u) \colon \Chu(\C C, \gamma) \to
\Chu(\C C, \delta).$$

If $F \colon \C C\to \C D$, a left $F$-coalgebra is a triplet
$\big(\gamma \in C_0, \delta \in \C D_0, \phi \colon F_0(\gamma) \to \delta\big)$.
If $G \colon \C D \to \C C$, a right $G$-coalgebra is a triplet
$\big(\gamma \in C_0, \delta \in \C D_0, \phi \colon \gamma \to G_0(\delta)\big)$.
If $\C D = \C C$, a right $F$-coalgebra of the form 
$\big(\gamma \in C_0, \gamma \in \C D_0, \phi \colon \gamma \to F_0(\gamma)\big)$
is traditionally called an $F$-coalgebra. The relation between Chu spaces and coalgebras is studied 
by Abramsky in~\cite{Ab18}. 

%
%
%
%
%

\section{Boolean Chu representations}
\label{sec: boolechurepr}

The following Chu representation is standard. Recall that the category $\Top$ of topological spaces is not cartesian closed,
and hence we cannot use Proposition~\ref{prp: cccrepr} to represent it.

\begin{proposition}[Chu representation of $\Top$]\label{prp: reprtop}
The functor $E^{\Top} \colon \Top \to \Chu(\Set, \D 2)$, defined by
$$E^{\Top}_0(X, T) = (X, \in_{\mathsmaller{X,T}}, T),$$
$$\in_{\mathsmaller{X,T}} \colon X \times T \to \D 2,$$ 
\[ \in_{\mathsmaller{X,T}}(x, G) = \left\{ \begin{array}{ll}
                 1   &\mbox{, $x \in G$}\\
                 0             &\mbox{, $x \notin G$,}
                 \end{array}
          \right. \] 
$$E^{\Top}_1\big(f \colon (X, T) \stackrel{\cnt} \longrightarrow (Y, S)\big) = \big(f, \big[E_1^{\Top}(f)\big]^-\big)
\colon (X, \in_{\mathsmaller{X,T}}, T) \to (Y, \in_{\mathsmaller{Y,S}}, S),$$
$$f^{-1} = \big[E_1^{\Top}(f)\big]^- \colon S \to T, \ \ \ \ U \mapsto f^{-1}(U),$$
is a strict Chu representation of $\Top$ into $\Chu(\Set, \D 2)$.
\end{proposition}

Notice that although the proof of the previous proof is constructive, the definition of $\in_{X,T}$ is classical. One can show classically
that the Chu space $(X, \in_{\mathsmaller{X,T}}, T)$ is separable if and only if the topology $T$ is $T_0$. Clearly,
$(X, \in_{\mathsmaller{X,T}}, T)$ is always extensional.
The special properties of a topology $T$ on a set $X$ play no role in the above definitions i.e., this representation applies 
to more general categories. E.g., a classical Chu representation $E^{\Set} \colon \Set \to \Chu(\Set, \D 2)$ is defined similarly by
$$E^{\Set}_0(X) = (X, f_X, \C P(X)),$$
$$f_X \colon X \times \C P(X) \to \D 2,$$ 
\[ f_X(x, A) = \left\{ \begin{array}{ll}
                 1   &\mbox{, $x \in A$}\\
                 0             &\mbox{, $x \notin A$,}
                 \end{array}
          \right. \] 
$$E^{\Set}_1\big(f \colon X \to Y\big) = \big(f, f^{-1}\big).$$
If we consider the full embedding $\Delta \colon \Set \to \Top$, where
$\Delta_0(X) = (X, \C P(X))$ and $\Delta_1(f \colon X \to Y) = f$,
the following triangle commutes
\begin{center}
\begin{tikzpicture}

\node (E) at (0,0) {$\Top$};
\node[right=of E] (S) {};
\node[right=of S] (T) {$\Chu(\Set, \D 2)$.};
\node[below=of E] (A) {$\Set$};

\draw[->] (E)--(T) node [midway,above] {$E^{\Top}$};
\draw[->] (A)--(E) node [midway,left] {$\Delta$};
\draw[->] (A)--(T) node [midway,right] {$ \ \ E^{\Set}$};

\end{tikzpicture}
\end{center}

%
%
%
%

For all notions mentioned next we refer to~\cite{SW12}, chapter 6. Recall that the Scott topology is Hausdorff, only in a trivial case, and hence it is not completely regular.

\begin{definition}\label{def: infcat}
Let  $\Inf$ be the category of information systems $(X, \Con_X, \vdash_X)$ together with morphisms 
$r \colon (X, \Con_{\mathsmaller{X}}, \vdash_{\mathsmaller{X}}) \to (Y, \Con_{\mathsmaller{Y}}, \vdash_{\mathsmaller{Y}})$
the approximable mappings i.e.,
appropriate relations $r \subseteq \Con_{\mathsmaller{X}} \times Y$. 
If $s \colon (Y, \Con_{\mathsmaller{Y}}, \vdash_{\mathsmaller{Y}}) \to (Z, \Con_{\mathsmaller{Z}}, 
\vdash_{\mathsmaller{Z}})$, the composition $s \circ r$ is defined by
$$A (s \circ r) z :\TOT \exists_{B \in \Con_B}\big(A r B \ \& \ B s z\big).$$
Moreover, $1_{(X, \Con_{\mathsmaller{X}}, \vdash_{\mathsmaller{X}})} = \ \vdash_{\mathsmaller{X}}$. Let 
$|X|$ be the set of ideals of $(X, \Con_{\mathsmaller{X}}, \vdash_{\mathsmaller{X}})$ and $S_{\mathsmaller{X}}$ the 
Scott topology on $|X|$ that has the sets
$\C O_A = \{J \in |X| \mid A \subseteq J\}$, 
where $A \in \Con_{\mathsmaller{X}}$, as a base.
\end{definition}

To show that $\vdash_X (X, \Con_{\mathsmaller{X}}, \vdash_{\mathsmaller{X}}) \to (X, \Con_{\mathsmaller{X}},
\vdash_{\mathsmaller{X}})$ we use the definition of an information 
system. To show that $1_{(X, \Con_{\mathsmaller{X}}, \vdash_{\mathsmaller{X}})} = \ \vdash_{\mathsmaller{X}}$ 
we use the definition of composition of approximable mappings.

\begin{proposition}[Chu represenation of $\Inf$]\label{prp: infchu}
 The functor $S \colon \Inf \to \Top$, where
 $$S_0(X, \Con_{\mathsmaller{X}}, \vdash_{\mathsmaller{X}}) = \big(|X|, S_{\mathsmaller{X}}\big),$$
 $$S_1\big(r \colon (X, \Con_{\mathsmaller{X}}, \vdash_{\mathsmaller{X}}) \to (Y, \Con_{\mathsmaller{Y}}, 
 \vdash_{\mathsmaller{Y}})\big) = |r| \colon |X| \to |Y|,$$
 $$|r|(J) = \big\{y \in Y \mid \exists_{J{'} \subseteq^{\fin} J}\big(J{'} r y\big)\big\},$$
 is a full embedding of $\Inf$ into $\Top$. Consequently, $E^{\Top} \circ S \colon \Inf \to \Chu(\Set, \D 2)$ is a
 a strict Chu representation of $\Inf$ into $\Chu(\Set, \D 2)$.
 
\end{proposition}

\begin{proof}
First we show that $|\vdash_X| = \id_{|X|}$. If $J \in |X|$, then 
$$|\vdash_{\mathsmaller{X}}|(J) = \big\{x \in X \mid \exists_{J{'} \subseteq^{\fin} J}\big(J{'} \vdash_{\mathsmaller{X}}
x\big)\big\}.$$
If $x \in |\vdash_{\mathsmaller{X}}|(J)$, then $J{'} \vdash_{\mathsmaller{X}} x$, for some $J{'} \subseteq^{\fin} J$,
hence $x \in \overline{J} = J$.
If $x \in J$, then $\{x\} \vdash_{\mathsmaller{X}} x$, and hence $x \in |\vdash_{\mathsmaller{X}}|(J)$. 
The equality $|r \circ s| = |r| \circ |s|$ is
straightforward to show. $S$ is full, as if $f \colon |X| \to |Y|$, then $f = |r_f|$, where $A r_f y :\TOT y \in 
f\big(\overline{A}\big)$. $S$ is injective on arrows; if $|r| = |s|$, then $r = r_{|r|} = r_{|s|} = s$. To show that $S$ is
injective on objects, we suppose that $\big(|X|, S_{\mathsmaller{X}}\big) = \big(|Y|, S_{\mathsmaller{X}}\big)$ 
and we show that 
$(X, \Con_{\mathsmaller{X}}, \vdash_{\mathsmaller{X}}) = (Y, \Con_{\mathsmaller{Y}}, \vdash_{\mathsmaller{Y}})$. 
If $x \in X$, then $\overline{\{x\}} \in |Y|$, hence $\overline{\{x\}}
\subseteq Y$, and consequently $x \in Y$. Similarly, we get $Y \subseteq X$. If $A \in \Con_{\mathsmaller{X}}$, then 
$$\overline{A}^{\mathsmaller{X}} = \{x \in X \mid A \vdash_{\mathsmaller{X}} x\} \in |Y|.$$
As $A \subseteq^{\fin} \overline{A}^{\mathsmaller{X}} \in |Y|$, we get $A \in \Con_{\mathsmaller{Y}}$. 
Similarly, we get $\Con_{\mathsmaller{Y}} \subseteq \Con_{\mathsmaller{X}}$. If $A \vdash_{\mathsmaller{X}} x$, then 
$$\overline{A}^{\mathsmaller{Y}} = \{y \in Y \mid A \vdash_{\mathsmaller{Y}} y\} \in |Y| = |X|.$$
Hence, there is $I \in |X|$ such that $I = \overline{A}^{\mathsmaller{Y}}$. As $A \subseteq^{\fin} I$ and $I$ is 
deductively closed, we get $a \in \overline{A}^{\mathsmaller{Y}}$ i.e., $A \vdash_{\mathsmaller{Y}} a$. Similarly, we 
get $\vdash_{\mathsmaller{Y}} \ \subseteq \ \vdash_{\mathsmaller{X}}$.  
\end{proof}

As the category $\Inf$ is cartesian closed, then, according to Proposition~\ref{prp: cccrepr}, there is a normal Chu representation of $\Inf$, which avoids classical reasoning.


\section{Normal Chu representations}
\label{sec: normalchurepr}

We have seen already the normal Chu representation of $\Set$ through $E^{\Set, \D 2}$ into $\Chu(\Set, \D 2)$. 
Next we present the normal Chu representation of the category of Bishop spaces. The notion of Bishop space is a constructive, function-theoretic alternative to the set-based notion of topological space, which was introduced by Bishop in~\cite{Bi67}, revived by Bridges in~\cite{Br12} and elaborated by the author in~\cite{Pe15}-\cite{Pe19b} and~\cite{Pe20b}-\cite{Pe21e}.
For the sake of completeness we give next all necessary definitions related to the proof of a strict Chu representation of the category of Bishop spaces.

\begin{definition}\label{def: cont1}
If $X$ is a set and $\Real$ is the set of real numbers, we denote by $\mathbb{F}(X)$
the set of functions from $X$ to $\Real$, by $\D F^*(X)$ the bounded elements of $\D F(X)$,
and by $\Const(X)$\index{$\Const(X)$} the 
subset of $\mathbb{F}(X)$ of all constant functions on $X$. 
If $a \in \Real$, we denote by $\overline{a}^X$\index{$\overline{a}^X$}
the constant function on $X$ with value $a$. We denote by $\Nat^+$ the set of non-zero natural numbers.  
A function $\phi: \Real \rightarrow \Real$ is called \textit{Bishop continuous}, or simply continuous,
if for every $n \in \Nat^+$ there is a function\index{$\omega_{\phi,n}$} $\omega_{\phi, n}:
\mathbb{R}^{+} \rightarrow \mathbb{R}^{+}$,
$\epsilon \mapsto \omega_{\phi, n}(\epsilon)$, which is called a \textit{modulus 
of continuity}\index{modulus of (uniform) continuity} of $\phi$ on $[-n, n]$, such that the following 
condition is satisfied
\[\forall_{x, y \in [-n, n]}(|x - y| < \omega_{\phi, n}(\epsilon) \Rightarrow 
|\phi(x) - \phi(y)| \leq \epsilon),\]
for every $\epsilon > 0$ and every $n \in \Nat^+$. We denote by $\BR$\index{$\BR$} the set of continuous functions 
from $\Real$ to $\Real$, which is equipped with the pointwise equality inherited from $\D F(\Real)$.

\end{definition}

\begin{definition}\label{def: notation1}
If $X$ is a set, $f, g \in \mathbb{F}(X)$, $\epsilon > 0$, and $\Phi \subseteq \mathbb{F}(X)$, 
let\index{$U(X; \Phi, g, \epsilon)$}  
\[ U(X; g, f, \epsilon) :\TOT \forall_{x \in X}\big(|g(x) - f(x)| \leq \epsilon\big),\]
\[ U(X; \Phi, f) :\TOT \forall_{\epsilon > 0}\exists_{g \in \Phi}\big(U(g, f, \epsilon)\big).\]
If the set $X$ is clear from the context, we write simply $U(f, g, \epsilon)$\index{$U(f,g, \epsilon)$} 
and\index{$U(\Phi, f)$} $U(\Phi, f)$, respectively.\index{$U(X; \Phi, f)$}
We denote by $\Phi^*$ the bounded elements of $\Phi$, and its uniform 
closure\index{uniform closure} $\overline{\Phi}$ is defined by\index{$\overline{\Phi}$} 
\[ \overline{\Phi} := \{f \in \D F(X) \mid U(\Phi, f)\}.\]
\end{definition}

A Bishop topology on $X$ is a certain subset of $\D F(X)$. As the Bishop topologies considered here
are all extensional\footnote{If $X$ is a set and $P$ is an extensional property on $X$ i.e., $P(x) \ \& \ x =_X y \To P(y)$,
the extensional subset $X_P$ of $X$ is defined by separation, $X_P = \{x \in X \mid P(x)\}$, its equality is inherited by that 
of $X$ and the embedding of $X_P$ into $X$ is defined by the identity rule (see~\cite{Pe20}, Definition 2.2.3).}
subsets of $\D F(X)$, we do not mention the embedding $i_F^{\D F(X)} \colon F \eto \D F(X)$, 
which is given in all cases by the identity map-rule. The uniform closure $\overline{\Phi}$ of $\Phi$ is 
an extensional subset of $\D F(X)$.

\begin{definition}\label{def: bishopspace}
A \textit{Bishop space}\index{Bishop space} is a pair $\C F := (X, F)$, where $F$ is an extensional subset of $\D F(X)$,
which is called a \textit{Bishop topology}, or a \textit{topology}\index{Bishop topology}
of functions on $X$, that satisfies the following conditions:\\[1mm]
$(\BS_1)$ If $a \in \Real$, then $\overline{a}^X \in F$.\\[1mm]
$(\BS_2)$ If $f, g \in F$, then $f + g \in F$.\\[1mm]
$(\BS_3)$ If $f \in F$ and $\phi \in \Bic(\Real)$, then $\phi \circ f \in F$
\begin{center}
\begin{tikzpicture}

\node (E) at (0,0) {$X$};
\node[right=of E] (F) {$\Real$};
\node[below=of F] (A) {$\Real$.};

\draw[->] (E)--(F) node [midway,above] {$f$};
\draw[->] (E)--(A) node [midway,left] {$F \ni \phi \circ f \ $};
\draw[->] (F)--(A) node [midway,right] {$\phi \in \BR$};

\end{tikzpicture}
\end{center}
$(\BS_4)$ $\overline{F} = F$.
\end{definition}

If $\C F := (X, F)$ is a Bishop space, then $\C F^* := (X, F^*)$ is the Bishop space of 
bounded\index{$\C F^*$}\index{$F^*$} elements of $F$. The constant functions
$\Const(X)$ is the \textit{trivial}
topology on $X$, while $\D F(X)$ is the \textit{discrete} topology on $X$. Clearly, if $F$ is a topology on $X$,
then $\Const(X) \subseteq F \subseteq \D F(X)$, and the set of its bounded elements
$F^{*}$ is also a topology on $X$. It is straightforward to see that the pair $\C R := (\Real, \BR)$ is a
Bishop space, which we call the \textit{Bishop space of reals}. If $X$ is a metric space, the set 
$C_{p}(X)$ of all weakly continuous functions of 
type $X \rightarrow \Real$, as it is defined in~\cite{BB85}, p.76, is the set of pointwise continuous ones. 
It is easy to see that the pair $\mathcal{W}(X) = (X, C_{p}(X))$ 
is Bishop space. Bishop calls $C_{p}(X)$ the weak topology on $X$, but here we avoid this term, since in~\cite{Pe15} 
we use this term for the Bishop topology that corresponds to the weak topology of open sets,
and we call $C_{p}(X)$ the \textit{pointwise} topology on $X$. If $X$ is a compact metric space, 
the set $C_{u}(X)$ of all uniformly continuous
functions of type $X \rightarrow \mathbb{R}$ is a topology, called by Bishop the 
\textit{uniform} topology on $X$. We call $\mathcal{U}(X) = (X, C_{u}(X))$ the uniform space. 
If $X$ is a locally compact metric space, the set $\Bic(X)$ of 
Bishop continuous functions from $X$ to $\Real$ i.e., uniformly continuous on 
every\footnote{As in the case of $\BR$, it seems that this definition requires quantification over the
power set of $X$ i.e., 
$$\BX(f) \TOT \forall_{B \in \mathcal{P}(X)}(\mbox{bounded}(B) 
\To f_{|B} \ \mbox{is uniformly continuous}).$$ 
A bounded subset $B$ of an inhabited metric space $X$ is a triplet 
$(B, x_{0}, M)$, where $x_{0} \in X, B \subseteq X$, and $M > 0$ is a bound for $B \cup \{x_{0}\}$.
To avoid such a quantification, if $x_{0}$ inhabits $X$, then for every bounded subset
$(B, x_{0}{'}, M)$ of $X$ we have that there is some $n \in \Nat$ such that $n > 0$ and 
$B \subseteq [d_{x_{0}} \leq \overline{n}] = \{x \in X \mid d(x_{0}, x) \leq n\}$. 
If $x \in B$, then $d(x, x_{0}) \leq d(x, x_{0}{'}) + d(x_{0}{'}, x_{0}) \leq M + d(x_{0}{'}, x_{0})$, therefore
$x \in [d_{x_{0}} \leq \overline{n}]$, for some $n > M + d(x_{0}{'}, x_{0})$. Hence, 
$$\BX(f) \TOT \forall_{n \in \Nat}(f_{|[d_{x_{0}} \leq \overline{n}]} \ \mbox{is uniformly continuous}),$$
since $[d_{x_{0}} \leq \overline{n}] = \{x \in d(x_0, x) \leq n\}$ is trivially a bounded subset of $X$.}
bounded subset of $X$, is a Bishop topology on $X$.

A Bishop topology $F$ is a ring and a lattice; since $|\id_{\Real}| \in \Bic(\Real)$, where $\id_{\Real}$ is the 
identity function on $\Real$, by BS$_{3}$ we get that if $f \in F$ then $|f| \in F$. 
By BS$_{2}$ and BS$_{3}$, and using the following equalities  
\[ f{\cdot}g = \frac{(f + g)^{2} - f^{2} - g^{2}}{2} \in F,\]
\[ f \vee g = \max\{f, g\} = \frac{f + g + |f - g|}{2} \in F, \]  
\[ f \wedge g = \min\{f, g\} = \frac{f + g - |f - g|}{2} \in F,\]
we get similarly that if $f, g \in F$, then $f{\cdot}g, f \vee g, f \wedge g \in F$.
Turning the definitional clauses of a Bishop topology into inductive rules, Bishop defined in~\cite{Bi67}, p.~72,
the least topology including a given subbase $F_{0}$. This inductive definition, which 
is also found in~\cite{BB85}, p.~78, is crucial to the definition of new Bishop topologies from given ones.

\begin{definition}\label{def: bishop}
The category of Bishop spaces $\Bis$ is 
the subcategory of $\Aff(\Set, \Real)$ with objects pairs $(X, F)$ such that $F \subseteq \D F(X)$ is a Bishop 
topology on $X$. 
\end{definition}

Consequently, if $\C F := (X, F)$ and $\C G = (Y, G)$ are Bishop spaces, a function $h: X \rightarrow Y$ is a morphism from 
$\C F$ to $\C G$ in $\Bis$, which is called
a \textit{Bishop morphism}, if $\forall_{g \in G}(g \circ h \in F)$ 
\begin{center}
\begin{tikzpicture}

\node (E) at (0,0) {$X$};
\node[right=of E] (F) {$Y$};
\node[below=of F] (A) {$\Real$.};

\draw[->] (E)--(F) node [midway,above] {$h$};
\draw[->] (E)--(A) node [midway,left] {$F \ni g \circ h \ $};
\draw[->] (F)--(A) node [midway,right] {$g \in G$};

\end{tikzpicture}
\end{center}
We denote by $\Mor(\C F, \C G)$\index{$\Mor(\C F, \C G)$} the set of Bishop morphisms 
from $\C F$ to $\C G$. As $F$ is an extensional subset of $\D F(X)$, $\Mor(\C F, \C G)$ is an
extensional subset of $\D F(X,Y)$. Similarly to $\Top$, the category $\Bis$ is not cartesian closed. 
The following Chu-representation of Bishop spaces is completely constructive, and its proof is equally simple to the 
proof of Proposition~\ref{prp: reprtop}.

\begin{proposition}[Chu representation of $\Bis$]\label{prp: reprbish}
The functor $E^{\Bis} \colon \Bis \to \Chu(\Set, \Real)$, defined by
$$E^{\Bis}_0(X, F) = (X, \ev_{\mathsmaller{X,F}}, F),$$
$$\ev_{\mathsmaller{X,F}} \colon X \times F \to \D R,$$ 
$$\ev_{\mathsmaller{X,F}}(x, f) = f(x),$$
$$E^{\Bis}_1\big(h \colon (X, F) \longrightarrow (Y, G)\big) = \big(h, h^*\big)
\colon (X, \ev_{\mathsmaller{X,F}}, F) \to (Y, \ev_{\mathsmaller{Y,G}}, G),$$
$$h^* \colon G \to F, \ \ \ \ h^*(g) = g \circ h,$$
is a strict Chu representation of $\Bis$ into $\Chu(\Set, \D R)$.
\end{proposition}

\begin{proof}
First we show that $\big(h, h^*\big)
\colon (X, \ev_{\mathsmaller{X,F}}, F) \to (Y, \ev_{\mathsmaller{Y,G}}, G)$ i.e., the following 
rectangle commutes
\begin{center}
\begin{tikzpicture}

\node (E) at (0,0) {$X \times G$};
\node[right=of E] (T) {};
\node[right=of T] (F) {$X \times F$};
\node[below=of E] (A) {$Y \times G$};
\node[below=of F] (S) {$\ \D R$};

\draw[->] (E)--(A) node [midway,left] {$\mathsmaller{h \times \id_{G}}$};
\draw[->] (F)--(S) node [midway,right] {$\ev_{\mathsmaller{X,F}}$};
\draw[->] (E)--(F) node [midway,above] {$\mathsmaller{\id_{X} \times h^*}$};
\draw[->] (A)--(S) node [midway,below] {$\mathsmaller{\ev_{\mathsmaller{Y,G}}}$};

\end{tikzpicture}
\end{center}
\begin{align*}
 \ev_{\mathsmaller{X,F}}\big((\id_{X} \times h^*)(x, g)\big) & = \ev_{\mathsmaller{X,F}}\big(x, g \circ h\big)\\
 & = g(h(x))\\
 & = \ev_{\mathsmaller{Y,G}}\big(h(x), g\big)\\
 & =  \ev_{\mathsmaller{Y,G}}\big((h \times \id_G)(x, g)\big).
\end{align*}
It is immediate to show that $E^{\Bis}$ is a functor, which is injective on objects and arrows. Next we show that
$E^{\Bis}$ is full. Let $\big(\phi^+, \phi^-\big) \colon 
(X, \ev_{\mathsmaller{X,F}}, F) \to (Y, \ev_{\mathsmaller{Y,G}}, G)$ i.e., $\phi^+ \colon X \to Y$ and 
$\phi^- \colon Y \to X$ such that the following rectangle commutes
\begin{center}
\begin{tikzpicture}

\node (E) at (0,0) {$X \times G$};
\node[right=of E] (T) {};
\node[right=of T] (F) {$X \times F$};
\node[below=of E] (A) {$Y \times G$};
\node[below=of F] (S) {$\D R$};

\draw[->] (E)--(A) node [midway,left] {$\mathsmaller{\phi^+ \times \id_{G}}$};
\draw[->] (F)--(S) node [midway,right] {$\ev_{\mathsmaller{X,F}}$};
\draw[->] (E)--(F) node [midway,above] {$\mathsmaller{\id_{X} \times \phi^-}$};
\draw[->] (A)--(S) node [midway,below] {$\ev_{\mathsmaller{Y,G}}$};

\end{tikzpicture}
\end{center}
\begin{align*}
 \ev_{\mathsmaller{X,F}}\big((\id_{X} \times \phi^-)(x, g)\big) & = \ev_{\mathsmaller{X,F}}\big(x, \phi^-(g)\big)\\
 & = \big[\phi^-(g)\big](x))\\
 & = g\big(\phi^+(x)\big)\\
 & = \ev_{\mathsmaller{Y,G}}\big(\phi^+(x), g\big)\\
 & =  \ev_{\mathsmaller{Y,G}}\big((\phi^+ \times \id_G)(x, g)\big).
\end{align*}
From the resulting equality $F \ni \phi^-(g) = g \circ \phi^+$, and since $g \in G$ is arbitrary,
we conclude that $\phi^+ \in \Mor(\C F, \C G)$. By the same equality we also get $\phi^- = \big(\phi^+\big)^*$, since, 
if $g \in G$, we have that
$$\big[\big(\phi^+\big)^*\big(g)](x) = g\big(\phi^+(x)\big) = \big[\phi^-(g)\big](x)).$$
Hence, $E^{\Bis}_1\big(\phi^+) = \big(\phi^+, \phi^-\big)$. 
\end{proof}

In~\cite{Pe15} the mapping $h^* = \big[E_1^{\Bis}(h)\big]^- \colon G \to F$ is the ring homomorphism 
induced by $h \in \Mor(\C F, \C G)$.
Let the Chu space $(X, \ev_{\mathsmaller{X,F}}, F)$, and by Definition~\ref{def: extsep} let $\widehat{\ev_{\mathsmaller{X,F}}} \colon X \to (F \to \Real)$
with $\widehat{\ev_{\mathsmaller{X,F}}}(x) = \widehat{x}$.
Consequently, the Chu space $(X, \ev_{\mathsmaller{X,F}}, F)$ 
is separable if and only if $F$ \textit{separates the points} of $X$:
\begin{align*}
 \widehat{x} =_{\D F(F, \Real)} \widehat{x{'}} & :\TOT \forall_{f \in F}\big(\widehat{x}(f) =_{\Real} \widehat{x{'}}(f)\big)\\
 & \TOT \forall_{f \in F}\big(f(x) =_{\Real} f(x{'})\big)\\
 & \TOT x =_X x{'}.
\end{align*}
If $\widecheck{\ev_{\mathsmaller{X,F}}} \colon F \to (X \to \Real)$
with $\widecheck{\ev_{\mathsmaller{X,F}}}(x) = \widecheck{f}$, then $(X, \ev_{\mathsmaller{X,F}}, F)$ 
is always extensional. Clearly, all these proofs concerning the Chu space $(X, \ev_{\mathsmaller{X,F}}, F)$ are constructive.

As in the case of the classical Chu representation of $\Top$, the Chu representation of $\Bis$ does not involve the 
special properties of a Bishop topology $F$ and it can be applied to other categories too.
The functor $C^{\Top} \colon \Top \to \Chu(\Set, \Real)$ defined by 
$$C^{\Top}_0(X, T) = (X, \ev_{X}, C(X)),$$
$$\ev_{X} \colon X \times C(X) \to \D R,$$ 
$$\ev_{X}(x, f) = f(x),$$
$$C^{\Top}_1\big(h \colon (X, T) \stackrel{\cnt} \longrightarrow (Y, S)\big) = \big(h, h^*\big)
\colon (X, \ev_{X}, C(X)) \to (Y, \ev_{Y}, C(Y)),$$
$$h^* \colon C(Y) \to C(X), \ \ \ \ h^*(g) = g \circ h,$$
is only an embedding of $\Top$ into $\Chu(\Set, \D R)$. To show that $C^{\Top}$ is full, one needs to show that if 
$\big(\phi^+, \phi^-\big) \colon (X, \ev_{X}, C(X)) \to (Y, \ev_{Y}, C(Y))$, then $\phi^+ \in C(X, Y)$. What we can show
only is that $\phi^-(g) = g \circ \phi^+ \in C(X)$, for every $g \in C(Y)$, something which does not imply, in general, that 
$\phi^+ \in C(X, Y)$. 
One can show that $\phi^+ \in C(X, Y)$, if $Y$ is completely regular i.e., a Hausdorff space $Y$ such that every closed
set $F$ and a point $y \notin F$ are separated by an element of $C(Y)$. Let $\crTop$ be the full subcategory of completely
regular topological spaces. It is not a coincidence that such a result holds (classically), as one can show classically that the 
canonical topology of open sets induced by some Bishop topology is completely regular. From the point of view of 
the theory of rings of continuous functions, the restriction to $\crTop$ is not a loss of generality, as 
for every topological space $X$ there is a completely regular space $\rho X$ such that the ring $C(X)$ is isomorphic to 
$C(\rho X)$. Actually, $\crTop$ is a reflective subcategory of $\Top$ (see~\cite{He68} and~\cite{Wa74}), as for every 
topological space $(X, T)$ there is a completely regular space $(\rho X, \rho T)$ and a continuous surjection 
$\tau_X \colon X \to \rho X$ such that for every completely
regular space $(Y, S)$ and continuous function $f \colon X \to Y$ there is a unique continuous function 
$\rho f \colon \rho X \to Y$ such that the following triangle commutes
\begin{center}
\begin{tikzpicture}

\node (E) at (0,0) {$X$};
\node[right=of E] (T) {};
\node[right=of T] (F) {$\rho X$};
\node[below=of F] (A) {$Y$.};

\draw[->] (E)--(F) node [midway,above] {$\tau_X$};
\draw[->] (F)--(A) node [midway,right] {$\rho f$};
\draw[->] (E)--(A) node [midway,left] {$f \ \ $};

\end{tikzpicture}
\end{center}

\begin{proposition}[Chu representation of $\crTop$]\label{prp: chureprcrtop}
 The functor $E^{\crTop} \colon \crTop \to \Chu(\Set, \Real)$, where 
$$E^{\crTop}_0(X, T) = (X, \ev_{X}, C(X)),$$
$$\ev_{X} \colon X \times C(X) \to \D R,$$ 
$$\ev_{X}(x, f) = f(x),$$
$$E^{\crTop}_1\big(h \colon (X, T) \stackrel{\cnt} \longrightarrow (Y, S)\big) = \big(h, h^*\big)
\colon (X, \ev_{X}, C(X)) \to (Y, \ev_{Y}, C(Y)),$$
$$h^* \colon C(Y) \to C(X), \ \ \ \ h^*(g) = g \circ h,$$
is a strict representation of $\crTop$ into $\Chu(\Set, \Real)$.   
\end{proposition}

\begin{proof}
It suffices to show that $\phi^+ \in C(X, Y)$. A Hausdorff space is completely regular if and only if the family 
$$Z(X) = \{\zeta(f) \mid f \in C(X)\}, \ \ \ \zeta(f) = \{x \in X \mid f(x) = 0\},$$ 
of zero sets of $X$ is a base for the closed sets of $X$ i.e., every closed set in $X$ is the intersection of a family 
of zero sets of $X$ (see~\cite{GJ60}, p.~38). As
\begin{align*}
 \big(\phi^+\big)^{-1}\big(\zeta(g)\big) & = \{x \in X \mid \phi^+(x) \in \zeta(f)\}\\
 & = \{x \in X \mid g(\phi^+(x)) = 0\}\\
 & = \zeta(g \circ \phi^+),
\end{align*}
and $g \circ \phi^+ \in C(X)$, we conclude that $\big(\phi^+\big)^{-1}\big(\zeta(g)\big)$ is closed in $X$, hence $\phi^+$ 
is continuous.
\end{proof}

If $(X, T)$ is a topological space a subset $C \subseteq C(X)$ determines the topology $T$, if the weak topology of $C$
i.e., the smallest topology $\tau(C)$ that turns all elements of $C$ into continuous functions, is equal to $T$. If $(X, T)$
is Hausdorff, then $(X, T)$ is completely regular if and only if $\tau(C(X)) = T$ (see~\cite{GJ60}, p.~40). By the argument 
in the proof of Proposition~\ref{prp: chureprcrtop} one shows (see~\cite{GJ60}, p.~40) that if $C \subseteq C(Y)$ with
$\tau(C) = S$, then a function $\phi^+ \colon (X, T) \to (Y, S)$ is continuous if and only if $g \circ \phi^+ \in C(X)$,
for every $g \in C$. A generalisation of the proof of Proposition~\ref{prp: reprbish} follows next. Its proof is identical to the proof of Proposition~\ref{prp: reprbish}.

\begin{proposition}[Chu representation of $\Aff(\Set, X)$]\label{prp: churepraff} 
If $X$ is a set, the rule $(A, F) \mapsto (A, \ev_{\mathsmaller{A,F}}, F)$ defines a strict Chu representation of $\Aff(\Set, X)$
into $\Chu(\Set, X)$.
\end{proposition}

%
%
%

\section{A Chu representation of the category of subsets}
\label{sec: chusubsets}

Next we present the categorical in spirit notion of subset of a (Bishop) set.

\begin{definition}\label{def: subset}
Let $(X, =_X)$ be a set. A subset\index{subset} of $X$ is a pair $(A, i_A^X)$\index{$(A, i_A^X)$}, where $(A, =_A)$ is a set and 
$i_A^X \colon A \hookrightarrow X$ is an embedding (i.e., an injection) of $A$ into $X$.
If $(A, i_A^X)$ and $(B, i_B^X)$ are subsets of $X$, then $A$ is a \textit{subset}\index{subset} of $B$, in symbols 
$(A, i_A^X) \subseteq (B, i_B^X)$\index{$(A, i_A^X) \subseteq (B, i_B^X)$}, or simpler 
$A \subseteq B$\index{$A \subseteq B$},
if there is $f \colon A \to B$ such that the following diagram commutes
\begin{center}
\begin{tikzpicture}

\node (E) at (0,0) {$A$};
\node[right=of E] (B) {};
\node[right=of B] (F) {$B$};
\node[below=of B] (A) {$X$.};

\draw[->] (E)--(F) node [midway,above] {$f$};
\draw[right hook->] (E)--(A) node [midway,left] {$i_A^X \ $};
\draw[left hook->] (F)--(A) node [midway,right] {$\ i_B^X$};

\end{tikzpicture}
\end{center} 
In this case we also write $f \colon A \subseteq B$. Usually we write $A$ instead of $(A, i_A^X)$.
The totality of the subsets of $X$ is the \textit{powerset}\index{powerset} $\C P(X)$\index{$\C P(X)$} of $X$,
and it is equipped with the equality
\[ (A, i_A^X) =_{\C P(X)} (B, i_B^X) :\TOT A \subseteq B \ \& \ B \subseteq A. \]
If $f \colon A \subseteq B$ and $g \colon B \subseteq A$, we write $(f, g) \colon A =_{\C P(X)} B$. The category $\B {\C P}(X)$ of subsets of $X$ has objects the subsets of $X$ and morphisms functions $f \colon A \to B$ as above.
\end{definition}

Since the membership condition for $\C P(X)$ requires quantification over the open-ended totality $\D V_0$ of predicative sets (see~\cite{Pe20}, chapter 2), the totality $\C P(X)$ is a \textit{proper class}. It is immediate to show that $f \colon A \subseteq B$ is an embedding, and that the category $\B {\C P}(X)$ is thin.

\begin{proposition}[Chu-representation of $\B {\C P}(X)$]\label{prp: bishopchu1}
 If $(X, =_X)$ is a set, the functor
 $E^X \colon \C P(X) \to \Chu(\Set, X)$, defined by
 $$E_0^X\big(A, i_{A}^X\big) = \big(A, I_{A}^X, \D 1\big),$$
 $$I_A^X \colon A \times \D 1 \to X, \ \ \ I_A^X(a, 0) = i_{A}^X(a); \ \ \ a \in A,$$
 $$E_1^X\big(f \colon \big(A, i_{A}^X\big) \to \big(B, i_{B}^X\big)\big) = (f, \id_{\D 1}) \colon 
 \big(A, I_{A}^X, \D 1\big) \to \big(B, I_{B}^X, \D 1\big),$$
 is a strict Chu representation of $\C P(X)$ into $\Chu(\Set, X)$.
\end{proposition}

\begin{proof}
If $f \colon \big(A, i_{A}^X\big) \to \big(B, i_{B}^X\big)$, then by the commutativity of the following triangle we get 
the commutativity of the following rectangle
\begin{center}
\begin{tikzpicture}

\node (E) at (0,0) {$A \times \D 1$};
\node[right=of E] (T) {};
\node[right=of T] (F) {$A \times \D 1$};
\node[below=of E] (A) {$B \times \D 1$};
\node[below=of F] (S) {$X$};
\node[right=of F] (K) {$A$};
\node[right=of K] (L) {};
\node[right=of L] (M) {$B$};
\node[below=of L] (N) {$X$};

\draw[->] (E)--(A) node [midway,left] {$\mathsmaller{f \times \id_{\D 1}}$};
\draw[->] (F)--(S) node [midway,right] {$I_{A}^X$};
\draw[->] (E)--(F) node [midway,above] {$\mathsmaller{\id_{A} \times \id_{\D 1}}$};
\draw[->] (A)--(S) node [midway,below] {$I_{B}^X$};
\draw[left hook->,bend left] (K) to node [midway,above] {$f$} (M);
\draw[left hook->] (M) to node [midway,right] {$ \ i_{B}^X$} (N);
\draw[right hook->] (K)--(N) node [midway,left] {$i_{A}^X \ \ $};

\end{tikzpicture}
\end{center} 
thus $E_1^X(f) \colon  \big(A, I_{A}^X, \D 1\big) \to \big(B, I_{B}^X, \D 1\big)$. 
Clearly, $E^X$ is a functor injective on objects and arrows, hence an embedding. Moreover, by the commutativity of
the above rectangle we
get the commutativity of the above triangle. Hence, if $(f, \id_{\D 1}) \colon 
 \big(A, I_{A}^X, \D 1\big) \to \big(B, I_{B}^X, \D 1\big)$ in $\Chu(\Set, X)$, then $f \colon A \subseteq B$ in $\C P(X)$,
and hence $E$ is full.   
\end{proof}


The category of subsets of $X$ and its Chu representation are generalised to a $\CCC$ $\C C$ as follows.

\begin{definition}\label{def: sub}
 The category $\Sub(\C C, \gamma)$ of subobjects of $\gamma$ has 
 objects monomorphisms of $\C C$ with codomain $\gamma$ and a morphism $f \colon i \to j$, where $i \colon a \eto \gamma$ 
 and $j \colon b \eto \gamma$ is a a morphism $f \colon a \to b$ such that the following triangle commutes 
 \begin{center}
\begin{tikzpicture}

\node (E) at (0,0) {$a$};
\node[right=of E] (B) {};
\node[right=of B] (F) {$b$};
\node[below=of B] (A) {$\gamma$.};

\draw[->] (E)--(F) node [midway,above] {$f$};
\draw[right hook->] (E)--(A) node [midway,left] {$i \ $};
\draw[left hook->] (F)--(A) node [midway,right] {$ \ j$};

\end{tikzpicture}
\end{center} 
\end{definition}

It is immediate to show that $f$ is a monomorphism and that $\Sub(\C C, \gamma)$ is thin.

\begin{proposition}[Chu representation of $\Sub(\C C, \gamma)$]\label{prp: subrepr}
 The functor
 $E^{\Sub(\C C, \gamma)} \colon \Sub(\C C, \gamma) \to \Chu(\C C, \gamma)$, defined by
 $$E_0^{\Sub(\C C, \gamma)}\big(i \colon a \eto \gamma\big) = \big(a, i \circ \pr_a, 1\big),$$
 \begin{center}
\begin{tikzpicture}

\node (E) at (0,0) {$a \times 1$};
\node[right=of E] (T) {$a$};
\node[right=of T] (F) {$x$};

\draw[right hook->] (E)--(T) node [midway,above] {$\pr_a$};
\draw[right hook->] (T)--(F) node [midway,above] {$i$};

\end{tikzpicture}
\end{center}
 $$E_1^{\Sub(\C C, \gamma)}\big(f \colon i \to j\big) = (f, 1_1) \colon \big(a, i \circ \pr_a, 1\big) \to 
 \big(b, j \circ \pr_b, 1\big),$$
 is a strict Chu representation of $\Sub(\C C, \gamma)$ into $\Chu(\C C, \gamma)$.
\end{proposition}

\begin{proof}
The morphism $\pr_a$ is an iso, hence a mono. To show that $E_1^{\Sub(\C C, \gamma)}(f)
\colon  \big(a, i \circ \pr_a, 1\big) \to  \big(b, j \circ \pr_b, 1\big)$, we show that the following diagram commutes
\begin{center}
\begin{tikzpicture}

\node (E) at (0,0) {$\mathsmaller{a \times 1}$};
\node[right=of E] (T) {};
\node[right=of T] (F) {$\mathsmaller{a \times 1}$};
\node[below=of E] (A) {$\mathsmaller{b \times 1}$};
\node[below=of F] (S) {$\gamma$};

\draw[->] (E)--(A) node [midway,left] {$\mathsmaller{f \times 1_{1}}$};
\draw[->] (F)--(S) node [midway,right] {$\mathsmaller{i \circ \pr_a}$};
\draw[->] (E)--(F) node [midway,above] {$\mathsmaller{1_{a \times 1}}$};
\draw[->] (A)--(S) node [midway,below] {$\mathsmaller{j \circ \pr_b}$};

\end{tikzpicture}
\end{center}
%
%
%
%
%
%
$$i \circ \pr_a = (j \circ f) \circ \pr_a = j \circ (f \circ \pr_a)
= j \circ [\pr_b \circ (f \times 1_1)]
= (j \circ \pr_b) \circ (f^+ \times 1_1),$$
as the equality $f \circ \pr_a = \pr_b \circ (f \times 1_1)$ follows from the definition of $f \times 1_1$
\begin{center}
\begin{tikzpicture}

\node (E) at (0,0) {$b \ $};
\node[right=of E] (L) {};
\node[right=of L] (F) {$b \times 1 $};
\node[right=of F] (N) {};
\node[right=of N] (A) {$ \ 1$.};
\node[above=of F] (K) {};
\node[above=of K] (B) {$a \times 1$};
\node[above=of L] (T) {$a$};
\node[above=of N] (S) {$ 1$};

\draw[->] (B)--(T) node [midway,left] {$\pr_a \ $};
\draw[->] (F)--(E) node [midway,above] {$\pr_b$};
\draw[->] (T)--(E) node [midway,left] {$f \ \ $};
\draw[->] (B)--(S) node [midway,right] {$ \ \pr_1  $};
\draw[->,dashed] (B)--(F) node [midway,right] {$\mathsmaller{f \times 1_1}$};
\draw[->] (F)--(A) node [midway,above] {$\pr_1$};
\draw[->] (S)--(A) node [midway,right] {$\ \ 1_1$};

\end{tikzpicture}
\end{center}
If $\big(a, i \circ \pr_a, 1\big) = \big(b, j \circ \pr_b, 1\big)$, then $a = b$, and $i \circ \pr_a = j \circ p_a$.
As $\pr_a$ is a mono, we get $i = j$, and hence $E^{\Sub(\C C, \gamma)}$ is injective on objects. It is trivially 
injective on arrows. To show that it is full, let 
$(\phi^+, \phi^-) \colon  \big(a, i \circ \pr_a, 1\big) \to 
\big(b, j \circ \pr_b, 1\big)$. Clearly, $\phi^- = 1_1$. By the previous equalities we get
$i \circ \pr_a = (j \circ \phi^+) 
\circ \pr_a$, and since $\pr_a$ is a mono, 
$j \circ \phi^+ = i$ i.e., $\phi^+ \colon i \to j$ in 
$\Sub(\C C, \gamma)$.
\end{proof}

\section{A Chu representation of the category of complemented subsets}
\label{sec: chucompl}

%

\begin{definition}\label{def: apartness}
Let $(X, =_X)$ be a set. An \textit{inequality}\index{inequality} on $X$, or an 
\textit{apartness relation}\index{apartness relation} on $X$, is a relation $x \neq_X y$\index{$x \neq_X y$} such that 
the following conditions are satisfied:\\[1mm]
$(\Ap_1)$ $\forall_{x, y \in X}\big(x =_X y \ \& \ x \neq_X y \To \bot \big)$.\\
$(\Ap_2)$ $\forall_{x, y \in X}\big(x \neq_X y \To y \neq_X x\big)$.\\
$(\Ap_3)$ $\forall_{x, y \in X}\big(x \neq_X y \To \forall_{z \in X}(z \neq_X x \ \vee \ z \neq_X y)\big)$.\\[1mm]
We write $(X, =_X, \neq_X)$\index{$(X, =_X, \neq_X)$} to denote the equality-inequality structure of a
set\index{equality-inequality structure of a set} $X$. If $\big(A, i_A^X\big)$ is a subset of $X$, the canonical inequality on $A$ induced by $\neq_X$ is defined by 
$$a \neq_A a{'} :\TOT i_A^X(a) \neq_X i_A^X(a{'}),$$
for every $a, a{'} \in A$.
If $(Y, =_Y, \neq_Y)$ is a set with inequality, a function $f \colon X \to Y$ is called strongly extensional, if $f(x) \neq_Y f(x{'}) \To
x \neq_X x{'}$, for every $x, x{'} \in X$. 
\end{definition}

\begin{remark}\label{rem: apartness1}
An inequality relation $x \neq_X y$ is extensional on $X \times X$.
\end{remark}

\begin{proof}
If $x, y \in X$ such that $x \neq y$, and if $x{'}, y{'} \in X$ such that $x{'} =_X x$ 
and $y{'} =_X y$, we show that
$x{'} \neq y{'}$. By $(\Ap_3)$ we get $x{'} \neq x$, which is excluded from $(\Ap_1)$, or $x{'} \neq y$,
which has to be the case. Hence, $y{'} \neq x{'}$, or $y{'} \neq y$. Since the last option is excluded similarly,
we get $y{'} \neq x{'}$, hence $x{'} \neq y{'}$.
\end{proof}

An inequality on a set $X$ induces a positively defined notion of disjointness of subsets of $X$.

\begin{definition}\label{def: apartsubsets}
Let $(X, =_X, \neq_X)$ be a set, and $(A, i_A^X), (B, i_B^X) \subseteq X$. We say that 
$A$ and $B$ are disjoint with respect to $\neq_X$\index{disjoint subsets}, in symbols 
$A \Disj_{\mathsmaller{\neq_X}} B$\index{$A \Disj_{\mathsmaller{\neq}} B$}, if
\[ A \underset{\mathsmaller{\mathsmaller{\mathsmaller{\neq_X}}}} \Disj B : 
\TOT \forall_{a \in A}\forall_{b \in B}\big(i_A^X(a) \neq_X i_B^X(b) \big). \]
If $\neq_X$ is clear from the context, we only write $A \Disj_{_X} B$ or even $A \Disj B$\index{$A \Disj B$}.
\end{definition}

Clearly, if $A \Disj B$, then $A \cap B$ is not inhabited. 
The positive disjointness of subsets of $X$ induces the notion of a complemented subset of $X$, and
the negative notion of the complement of a set is avoided. We use bold letters to denote a complemented subset of a set.

\begin{definition}\label{def: complementedsubset}
A \textit{complemented subset}\index{complemented subset} of a set $(X, =_X, \neq_X)$ is a pair
$\B A := (A^1, A^0)$\index{$\B A := (A^1, A^0)$}, 
where $(A^1, i_{A^1}^X)$ and $(A^0, i_{A^0}^X)$ are subsets of $X$ such that $A^1 \Disj A^0$.
If 
$\Dm(\B A) := A^1 \cup A^0$\index{$\Dm(\B A)$} is the domain of $\B A$\index{domain of a complemented subset},
the 
\textit{indicator function}\index{indicator function of a complemented subset}, or
\textit{characteristic function}\index{characteristic function of a complemented subset}, of
$\B A$ is the operation $\chi_{\B A} : \Dm(\B A) \sto \D 2$ defined by
\[ \chi_{\B A}(x) := \left\{ \begin{array}{ll}
                 1   &\mbox{, $x \in A^1$}\\
                 0             &\mbox{, $x \in A^0$.}
                 \end{array}
          \right. \] 
Let $x \in \B A :\TOT x \in A^1$ and $x \notin \B A :\TOT x \in A^0$. If $\B A, \B B$ are complemented subsets of $X$, let
\[ \B A \subseteq \B B : \TOT A^1 \subseteq B^1 \ \& \ B^0 \subseteq A^0,  \]
Let $\C P^{\Disj}(X)$\index{$\C P^{\Disj}(X)$}\index{$\B A \subseteq \B B$} be their totality, 
equipped with the equality 
$\B A =_{\C P^{\mathsmaller{\Disj}} (X)} \B B : \TOT \B A \subseteq \B B \ \& \ \B B \subseteq \B A$.

\end{definition}

Clearly, $ \B A =_{\C P^{\mathsmaller{\Disj}} (X)} \B B \TOT A^1 =_{\C P(X)} B^1 \ \& \ A^0 =_{\C P(X)} B^0$. Notice that
if $f_1 \colon A^1 \subseteq B^1$ and $f_0 \colon B^0 \subseteq A^0$, then $f_1, f_0$ are strongly extensional functions.
E.g., if $f_1(a_1) \neq_{B^1} f_1(a_1{'})$, for some $a_1, a_1{'} \in A^1$, then from the definition of the canonical
inequality $\neq_{B^1}$ this means that $i_{B^1}^X\big(f_1(a_1)\big) \neq_{X} i_{B^1}^X\big(f_1(a_1{'})\big)$. By the 
extensionality of $\neq_X$ we get $i_{A^1}^X(a_1) \neq i_{B^1}^X(a_1{'}) :\TOT a_1 \neq_{A^1} a_1{'}$.

\begin{definition}\label{def: complcat}
If $(X, =_X, \neq_X)$ is a set, the category $\B {\C P}^{\Disj}(X)$ has objects the complemented subsets of $X$ and a morphism $f \colon \B A \to \B B$ is a pair $f = (f_1, f_0) \colon \B A \subseteq \B B$ i.e., 
$f_1 \colon A^1 \subseteq B^1$ and $f_0 \colon B^0 \subseteq A^0$. The unit morphism $1_{\B A}$ of $\B A$ is the 
pair $(\id_{A^1}, \id_{A^0})$, and if $g = (g_1, g_0) \colon \B B \subseteq \B C$, then
$g \circ f := (g_1 \circ f_1, f_0 \circ g_0)$ 
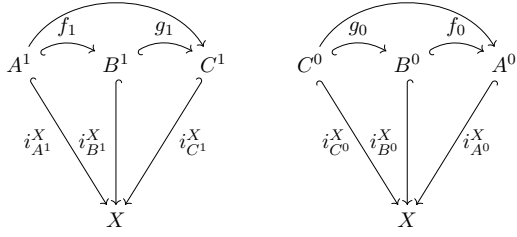
\begin{center}
\resizebox{7cm}{!}{%
\begin{tikzpicture}

\node (E) at (0,0) {$A^1$};
\node[right=of E] (F) {$B^1$};
\node[below=of F] (C) {};
\node[below=of C] (A) {$X$};
\node[right=of F] (K) {$C^1$};

\node[right=of K] (L) {$C^0$};
\node[right=of L] (M) {$B^0$};
\node[below=of M] (N) {};
\node[below=of N] (O) {$X$};
\node[right=of M] (P) {$A^0$};

\draw[left hook->,bend left] (E) to node [midway,above] {$f_1$} (F);
\draw[right hook->] (F) to node [midway,left] {$i_{B^1}^X$} (A);
\draw[right hook->] (E)--(A) node [midway,left] {$i_{A^1}^X \ $};
\draw[left hook->] (K)--(A) node [midway,right] {$ \ i_{C^1}^X$};
\draw[left hook->,bend left] (F) to node [midway,above] {$g_1$} (K);
\draw[->,bend left=60] (E) to node [midway,below] {} (K) ;

\draw[left hook->,bend left] (L) to node [midway,above] {$g_0$} (M);
\draw[right hook->] (M) to node [midway,left] {$i_{B^0}^X$} (O);
\draw[right hook->] (P)--(O) node [midway,right] {$i_{A^0}^X \ $};
\draw[left hook->] (L)--(O) node [midway,left] {$ \ i_{C^0}^X$};
\draw[left hook->,bend left] (M) to node [midway,above] {$f_0$} (P);
\draw[->,bend left=60] (L) to node [midway,below] {} (P) ;

\end{tikzpicture}
}
\end{center} 
\end{definition}

%
%
%

Clearly, the category $\B {\C P}^{\Disj}(X)$ is thin.

\begin{proposition}[Chu representation of $\B {\C P}^{\Disj}(X)$]\label{prp: bishopchu2}
 If $(X, =_X, \neq_X)$ is a set with an inequality, then the functor
 $E^{\B X} \colon \C P^{\Disj}(X) \to \Chu(\Set, X \times X)$, defined by
 $$E_0^{\B X}\big(A^1, i_{A^1}^X, A^0, i_{A^0}^X\big) = \big(A^1, i_{A^1}^X \times i_{A^0}^X, A^0\big),$$
 $$E_1^{\B X}\big((f^1, f^0) \colon \B A \to \B B\big) = (f^1, f^0) \colon 
 \big(A^1, i_{A^1}^X \times i_{A^0}^X, A^0\big) \to \big(B^1, i_{B^1}^X \times i_{B^0}^X, B^0\big),$$
 is a strict Chu representation of $\C P^{\Disj}(X)$ into $\Chu(\Set, X \times X)$.
\end{proposition}

\begin{proof}
  Let $i_{A^1}^X \times i_{A^0}^X 
 \colon A^1 \times A^0 \to X \times X$ where $\big[i_{A^1}^X \times i_{A^0}^X\big](a^1, a^0) = \big(i_{A^1}^X(a^1), 
 i_{A^0}^X(a^0)\big)$, for every $(a^1, a^0) \in A^1 \times A^0$. If $(f^1, f^0) \colon \B A \to \B B$, then 
 $(f^1, f^0) \colon \big(A^1, i_{A^1}^X \times i_{A^0}^X, A^0\big) \to \big(B^1, i_{B^1}^X \times i_{B^0}^X, B^0\big)$ 
 is a morphism in $\Chu(\Set, X \times X)$, as the commutativity of the following rectangle
\begin{center}
\begin{tikzpicture}

\node (E) at (0,0) {$\mathsmaller{A^1 \times B^0}$};
\node[right=of E] (T) {};
\node[right=of T] (F) {$\mathsmaller{A^1 \times A^0}$};
\node[below=of E] (A) {$\mathsmaller{B^1 \times B^0}$};
\node[below=of F] (S) {$\mathsmaller{X \times X}$};

\draw[->] (E)--(A) node [midway,left] {$\mathsmaller{f^1 \times \id_{B^0}}$};
\draw[->] (F)--(S) node [midway,right] {$\mathsmaller{i_{A^1}^X \times i_{A^0}^X}$};
\draw[->] (E)--(F) node [midway,above] {$\mathsmaller{\id_{A^1} \times f^0}$};
\draw[->] (A)--(S) node [midway,below] {$\mathsmaller{i_{B^1}^X \times i_{B^0}^X}$};

\end{tikzpicture}
\end{center}
follows from the commutativity of the following two triangles   
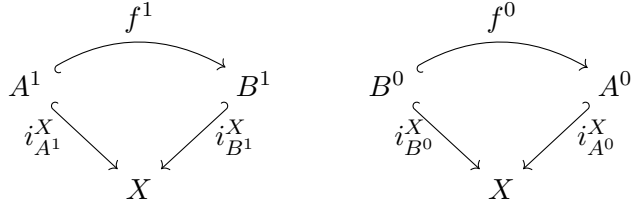
\begin{center}
\begin{tikzpicture}

\node (E) at (0,0) {$A^1$};
\node[right=of E] (B) {};
\node[right=of B] (F) {$B^1$};
\node[below=of B] (A) {$X$};

\node[right=of F] (L) {$B^0$};
\node[right=of L] (T) {};
\node[right=of T] (M) {$A^0$};
\node[below=of T] (O) {$X$};

\draw[left hook->,bend left] (E) to node [midway,above] {$f^1$} (F);
\draw[left hook->] (F) to node [midway,right] {$ \ i_{B^1}^X$} (A);
\draw[right hook->] (E)--(A) node [midway,left] {$i_{A^1}^X \ $};

\draw[left hook->,bend left] (L) to node [midway,above] {$f^0$} (M);
\draw[right hook->] (L) to node [midway,left] {$i_{B^0}^X$} (O);
\draw[left hook->] (M)--(O) node [midway,right] {$ \ i_{A^0}^X \ $};

\end{tikzpicture}
\end{center}
\begin{align*}
\big[\big(i_{A^1}^X \times i_{A^0}^X\big) \circ \big(1_{A^1} \times f^0\big)\big](a^1, b^0) & = \big[i_{A^1}^X \times 
i_{A^0}^X\big](a^1, f^0(b^0))\\
& = \big(i_{A^1}^X(a^1), i_{A^0}^X(f^0(b^0))\big)\\
& = \big(i_{B^1}^X(f^1(a^1)), i_{B^0}^X(b^0)\big)\\
& = \big[i_{B^1}^X \times i_{B^0}^X\big]\big(f^1(a^1), b^0\big)\\
& = \big[\big(i_{B^1}^X \times i_{B^0}^X\big) \circ \big(f^1 \times 1_{B^0}\big)\big](a^1, b^0).
\end{align*}
Clearly, $E^{\B X}$ is a functor injective on objects and arrows, hence an embedding. It is also full, 
as the above equalities also show that
the commutativity of the above rectangle implies the commutativity of the above triangles.
hence, if $(f^1, f^0) \colon \big(A^1,
i_{A^1}^X \times i_{A^0}^X, A^0\big) \to \big(B^1, i_{B^1}^X \times i_{B^0}^X, B^0\big)$ in $\Chu(\Set, X \times X)$, then 
$(f^1, f^0) \colon \B A \to \B B$. 
\end{proof}

Consequently, one can identify $\B {\C P}^{\Disj}(X)$ with the full subcategory
 of $\Chu(\Set, X \times X)$ with objects triplets $\big(A^1, i_{A^1}^X \times i_{A^0}^X, A^0\big)$, where
 $i_{A^1}^X \colon A^1 \eto X$ and $i_{A^0}^X \colon A^0 \eto X$ such that 
 $\forall_{a^1 \in A^1}\forall_{a^0 \in A^0}\big(i_{A^1}^X(a^1) \neq_X i_{A^0}^X(a^0)\big)$. Notice that the Chu category
 $\Chu(\Set, X \times X)$ ``captures'' the behavior of the morphisms in $\B {\C P}^{\Disj}(X)$, but not the positive 
 disjointness of $A^1, A^0$, as there are objects $(A, f, B)$ of $\Chu(\Set, X \times X)$, with $A \between B$; 
e.g., we may consider the triplet $(X, \id_{X \times X}, X)$.


\section{The generalised Chu construction over a ccc $\C C$ and an endofunctor}
\label{sec: chu2}

In order to Chu-represent categories like the category of predicates $\Pred$ and the category of complemented predicates 
$\Pred^{\neq}$, defined in the following two sections, respectively, we generalise the Chu construction. 
Actually, it is this embedding that shaped the ``right'' definition of the category $\Pred^{\neq}$, as, at first sight, more than one possible options exist.

\begin{definition}[The Chu construction over a ccc $\C C$ and an endofunctor]\label{def: chugen}
Let $\Gamma \colon \C C \to \C C$ an endofunctor on $\C C$.
The Chu category $\Chu(\C C, \Gamma)$ over $\C C$ and 
$\Gamma$ has objects quadruples $(x ; a, f, b)$, with $x, a, b \in C_0$ and $f \colon a \times b \to \Gamma_0(x) \in C_1$.
A morphism $\phi \colon (x; a, f, b) \to (y; c, g, d)$ in  $\Chu(\C C, \Gamma)$, or a Chu transform,
is a triplet $\phi = \big(\phi^0, \phi^+, \phi^-\big)$, where $\phi^0 \colon x \to y$, $\phi^+ \colon a \to c$ and $\phi^- 
\colon d \to b$ are in $C_1$ such that the following diagram commutes
\begin{center}
\begin{tikzpicture}

\node (E) at (0,0) {$\mathsmaller{a \times d} \  $};
\node[right=of E] (T) {};
\node[right=of T] (F) {$ \ \mathsmaller{a \times b}$};
\node[below=of F] (L) {$\mathsmaller{\Gamma_0(x)}$};
\node[below=of L] (S) {$\mathsmaller{\Gamma_0(y)}.$};
\node[left=of S] (K) {};
\node[left=of K] (A) {$\mathsmaller{c \times d}$};

\draw[->] (E)--(A) node [midway,left] {$\mathsmaller{\phi^+ \times 1_d}$};
\draw[->] (F)--(L) node [midway,right] {$\mathsmaller{f}$};
\draw[->] (E)--(F) node [midway,above] {$\mathsmaller{1_a \times \phi^-}$};
\draw[->] (L)--(S) node [midway,right] {$\mathsmaller{\Gamma_1(\phi^0)}$};
\draw[->] (A)--(S) node [midway,below] {$\mathsmaller{g}$};

\end{tikzpicture}
\end{center}
If $\theta = \big(\theta^0, \theta^+, \theta^-\big) \colon (y; c, g, d) \to (z; i, h, j)$, let
$\theta \circ \phi = \big(\theta^0 \circ \phi^0, \theta^+ \circ \phi^+, \phi^- \circ \theta^-\big)$
\begin{center}
\begin{tikzpicture}

\node (E) at (0,0) {$\mathsmaller{a \times d} \  $};
\node[right=of E] (T) {};
\node[right=of T] (F) {$ \ \mathsmaller{a \times b}$};
\node[below=of F] (L) {$\mathsmaller{\Gamma_0(x)}$};
\node[below=of L] (S) {$\mathsmaller{\Gamma_0(y)}$};
\node[left=of S] (K) {};
\node[left=of K] (A) {$\mathsmaller{c \times d}$};
\node[below=of S] (T) {$\mathsmaller{\Gamma_0(z)}$};
\node[below=of T] (U) {$  \mathsmaller{i \times j}.$};
\node[left=of U] (P) {};
\node[left=of P] (W) {$\mathsmaller{c \times j} \ $};
\node[right=of S] (C) {};
\node[right=of C] (X) {$\mathsmaller{a \times j}$};

\draw[->] (E)--(A) node [midway,left] {$\mathsmaller{\phi^+ \times 1_d}$};
\draw[->] (F)--(L) node [midway,right] {$\mathsmaller{f}$};
\draw[->] (E)--(F) node [midway,above] {$\mathsmaller{1_a \times \phi^-}$};
\draw[->] (L)--(S) node [midway,left] {$\mathsmaller{\Gamma_1(\phi^0)}$};
\draw[->] (A)--(S) node [midway,below] {$\mathsmaller{g}$};
\draw[->] (S)--(T) node [midway,left] {$\mathsmaller{\Gamma_1(\theta^0)}$};
\draw[->] (W)--(A) node [midway,left] {$\mathsmaller{1_c \times \theta^-}$};
\draw[->] (W)--(U) node [midway,below] {$\mathsmaller{\theta^+ \times 1_j}$};
\draw[->] (U)--(T) node [midway,right] {$\mathsmaller{h}$};
\draw[->] (X)--(F) node [midway,right] {$\mathsmaller{1_a \times (\phi^- \circ \theta^-)}$};
\draw[->] (X)--(U) node [midway,right] {$\mathsmaller{(\theta^+ \circ \phi^+) \times 1_j}$};
\draw[->,bend left=40] (L) to node [midway,right] {$\mathsmaller{\Gamma_1(\theta^0 \circ \phi^0)}$} (T);

\end{tikzpicture}
\end{center}
Moreover, $1_{(x; a, f, b)} = (1_x, 1_a, 1_b)$
\begin{center}
\begin{tikzpicture}

\node (E) at (0,0) {$\mathsmaller{a \times b} \ $};
\node[right=of E] (T) {};
\node[right=of T] (F) {$ \ \mathsmaller{a \times b}$};
\node[below=of F] (L) {$\mathsmaller{\Gamma_0(x)}$};
\node[below=of L] (S) {$\mathsmaller{\Gamma_0(x)}.$};
\node[left=of S] (K) {};
\node[left=of K] (A) {$\mathsmaller{a \times b}$};

\draw[->] (E)--(A) node [midway,left] {$\mathsmaller{1_a \times 1_b}$};
\draw[->] (F)--(L) node [midway,right] {$\mathsmaller{f}$};
\draw[->] (E)--(F) node [midway,above] {$\mathsmaller{1_a \times 1_b}$};
\draw[->] (L)--(S) node [midway,right] {$\mathsmaller{\Gamma_1(1_x) = 1_{\Gamma_0(x)}}$};
\draw[->] (A)--(S) node [midway,below] {$\mathsmaller{f}$};

\end{tikzpicture}
\end{center}
\end{definition}
To show that composition in $\Chu, (\C C, \Gamma)$ is well-defined, we show the commutativity of the above triangle
as follows:
\begin{align*}
 \Gamma_1(\theta^0 \circ \phi^0) \circ f \circ [1_a \times (\phi^- \circ \theta^-)] & =
  \Gamma_1(\theta^0) \circ \Gamma_1(\phi^0) \circ f \circ [1_a \times (\phi^- \circ \theta^-)]\\
 & \stackrel{(\ref{eq: cmc2})} = \Gamma_1(\theta^0) \circ \big[\Gamma_1(\phi^0) \circ f \circ (1_a \times \phi^-)\big] 
 \circ (1_a \times \theta^-)\\
 & = \Gamma_1(\theta^0) \circ  g \circ (\phi^+ \times 1_d) \circ (1_a \times \theta^-)\\
 & \stackrel{(\ref{eq: cmc3})} =  \Gamma_1(\theta^0) \circ  g \circ (1_c \times \theta^-) \circ (\phi^+ \times 1_j)\\
 & = \big[\Gamma_1(\theta^0) \circ  g \circ (1_c \times \theta^-)\big] \circ (\phi^+ \times 1_j)\\
 & = \big[h \circ (\theta^+ \times 1_j)\big] \circ (\phi^+ \times 1_j)\\
 & \stackrel{(\ref{eq: cmc4})} = h \circ \big[(\theta^+ \circ \phi^+) \times 1_j\big].
\end{align*}

\begin{proposition}\label{prp: relation}
 Let $\Gamma^{\gamma} \colon \C C \to \C C$ the constant 
 endofunctor with value $\gamma$ i.e., $\Gamma^{\gamma}_0(a) = \gamma$, for every $a \in C_0$, and $\Gamma^{\gamma}_1(f)
 = 1_{\gamma}$, for every $f \in C_1$. The functor $E^{\gamma} \colon \Chu(\C C, \gamma) \to \Chu(\C C, \Gamma^{\gamma})$, 
 defined by 
 $$E^{\gamma}_0(a, f, b) = (\gamma; a, f, b),$$
 $$E^{\gamma}_1\big(\big(\phi^+, \phi^-\big) \colon (a, f, b) \to (c, g, d)\big) = \big(1_{\gamma}, \phi^+, \phi^+\big) 
 \colon (\gamma; a, f, b) \to (\gamma; c, g, d),$$
 is an embedding of $\Chu(\C C, \gamma)$ into $\Chu(\C C, \Gamma^{\gamma})$. 
\end{proposition}

\begin{proof}
To show that $E^{\gamma}$ is a functor, it suffices to show that $\big(1_{\gamma}, \phi^+, \phi^+\big) 
 \colon (\gamma; a, f, b) \to (\gamma; c, g, d)$. This follows from the fact that the commutativity of the 
 following upper inner diagram implies the commutativity of the following outer diagram
\begin{center}
\begin{tikzpicture}

\node (E) at (0,0) {$\mathsmaller{a \times d} \  $};
\node[right=of E] (T) {};
\node[right=of T] (F) {$ \ \mathsmaller{a \times b}$};
\node[below=of F] (L) {$\mathsmaller{\Gamma_0(x)}$};
\node[below=of L] (S) {$\mathsmaller{\Gamma_0(y)}.$};
\node[left=of S] (K) {};
\node[left=of K] (A) {$\mathsmaller{c \times d}$};

\draw[->] (E)--(A) node [midway,left] {$\mathsmaller{\phi^+ \times 1_d}$};
\draw[->] (F)--(L) node [midway,right] {$f$};
\draw[->] (E)--(F) node [midway,above] {$\mathsmaller{1_a \times \phi^-}$};
\draw[->] (L)--(S) node [midway,right] {$\mathsmaller{1_{\gamma}}$};
\draw[->] (A)--(S) node [midway,below] {$\mathsmaller{g}$};
\draw[->] (A)--(L) node [midway,left] {$\mathsmaller{g} \ $};

\end{tikzpicture}
\end{center}
Clearly, $E^{\gamma}$ is injective on objects and arrows, hence it is an embedding.
\end{proof}

\begin{proposition}[The generalised local Chu functor]\label{prp: intfunctor}
The rule $\Chu^{\C C} \colon \Fun(\C C, \C C) \to \Cat$ defined by
$$\Chu^{\C C}_0(\Gamma) = \Chu(\C C, \Gamma),$$
$$\Chu^{\C C}_1(\eta \colon \Gamma \To \Delta) \colon \Chu(\C C, \Gamma) \to \Chu(\C C, \Delta),$$
$$\big[\Chu^{\C C}_1(\eta)\big]_0(x; a, f, b) = (x; a, \eta_x \circ f, b),$$
\begin{center}
\begin{tikzpicture}

\node (E) at (0,0) {$\mathsmaller{a \times b} $};
\node[right=of E] (T) {};
\node[right=of T] (F) {$\mathsmaller{\Gamma_0(x)}$};
\node[below=of F] (L) {$\mathsmaller{\Delta_0(x)}$};

\draw[MyBlue,->] (E)--(F) node [midway,above] {$\mathsmaller{f}$};
\draw[MyBlue,->] (F)--(L) node [midway,right] {$\mathsmaller{\eta_x}$};
\draw[->] (E)--(L) node [midway,above] {};

\end{tikzpicture}
\end{center}
$$\big[\Chu^{\C C})\big]_1\big(\phi^0, \phi^+, \phi^-\big) = \big(\phi^0, \phi^+, \phi^-\big),$$
is a functor. Moroever, if $\eta_x \colon \Gamma_0(x) \eto \Delta_0(x)$ is a mono, for every $x \in C_0$, then 
$\Chu^{\C C}_1(\eta)$ is a full embedding of $\Chu(\C C, \Gamma)$ into $\Chu(\C C, \Delta)$.
\end{proposition}

\begin{proof}
To show that $\Chu_1^{\C C}$ is a functor, it suffices to show that if $\big(\phi^0, \phi^+, \phi^-\big) \colon  
(x; a, f, b) \to (y; c, g, d)$ in $\Chu(\C C, \Gamma)$, then $\big(\phi^0, \phi^+, \phi^-\big) \colon  
(x; a, \eta \circ f, b) \to (y; c, \eta_y \circ g, d)$ in $\Chu(\C C, \Delta)$. This follows from the fact that
commutativity of the following upper, inner diagram implies the commutativity of the following outer diagram
\begin{center}
\begin{tikzpicture}

\node (E) at (0,0) {$\mathsmaller{a \times d} \  $};
\node[right=of E] (U) {};
\node[right=of U] (T) {};
\node[right=of T] (F) {$ \mathsmaller{a \times b}$};
\node[below=of F] (L) {$\mathsmaller{\Gamma_0(x)}$};
\node[below=of L] (R) {$\mathsmaller{\Delta_0(x)}$};
\node[below=of R] (S) {$ \ \ \mathsmaller{\Delta_0(y)}$};
\node[left=of S] (K) {$ \ \mathsmaller{\Gamma_0(y)} \ $};
\node[left=of K] (A) {$\mathsmaller{c \times d} \ $};

\draw[->] (E)--(A) node [midway,left] {$\mathsmaller{\phi^+ \times 1_d}$};
\draw[->] (F)--(L) node [midway,right] {$\mathsmaller{f}$};
\draw[->] (L)--(R) node [midway,right] {$\mathsmaller{\eta_x}$};
\draw[->] (L)--(K) node [midway,left] {$\mathsmaller{\Gamma_1(\phi^0)} \ $};
\draw[->] (E)--(F) node [midway,above] {$\mathsmaller{1_a \times \phi^-}$};
\draw[->] (R)--(S) node [midway,right] {$\mathsmaller{\Delta_1(\phi^0)}$};
\draw[->] (A)--(K) node [midway,below] {$\mathsmaller{g}$};
\draw[->] (K)--(S) node [midway,below] {$\mathsmaller{\eta_y}$};

\end{tikzpicture}
\end{center}
\begin{align*}
\Delta_1(\phi^0) \circ \eta_x \circ f \circ (1_a \times \phi^-) & =  
\big[\Delta_1(\phi^0) \circ \eta_x\big] \circ f \circ (1_a \times \phi^-)\\
& = \eta_y \circ \big[\Gamma_1(\phi^0) \circ f \circ (1_a \times \phi^-)\big]\\ 
& = \eta_y \circ g \circ \big(\phi^+ \times 1_d\big). 
\end{align*}
If $\eta_x \colon \Gamma_0(x) \eto \Delta_0(x)$ is a mono, for every $x \in C_0$, then $\Chu^{\C C}_1(\eta)$ is
injective on objects, and since it is trivially injective on arrows, it is an embedding. In this case, $\Chu^{\C C}_1(\eta)$
is also full, as the commutativity of the above outer diagram implies the commutativity of the above, upper, inner diagram.
As $\eta_y$ is a mono, the resulted equality 
$$\eta_y \circ \big[\Gamma_1(\phi^0) \circ f \circ (1_a \times \phi^-)\big] = 
\eta_y \circ g \circ \big(\phi^+ \times 1_d\big)$$
implies the equality $\Gamma_1(\phi^0) \circ f \circ (1_a \times \phi^-) = g \circ \big(\phi^+ \times 1_d$. 
\end{proof}

\begin{definition}\label{def: genchurepresent}
Let $\C C, \C D$ be categories and $F \colon C \to \C D$ a functor. If $\C D$ is a generalised Chu category and 
$F$ is a representation, we call $F$ a generalised 
Chu representation. We call a generalised Chu representation $F$ strict, if $F$ is injective on arrows. 
 
\end{definition}

%
%

\section{The generalised global Chu functor}
\label{sec: genchufunctor}

The following fact is the generalised analogue to Lemma~\ref{lem: lemextfunctor}.

\begin{lemma}\label{lem: Lemextfunctor}
 Let $\C C, \C D$ be cartesian closed categories, $\Gamma \colon \C C \to \C C, \Delta \colon \C D \to \C D$, 
 $F \colon  \C C \to \C D$ such that 
 $F$ preserves products with $(F_{ab})_{a, b \in C_0}$ the canonical isomorphisms 
 of $F$, and let $\eta \colon F \circ \Gamma \To \Delta \circ F$
 \begin{center}
\begin{tikzpicture}

\node (E) at (0,0) {$\mathsmaller{\C C} $};
\node[right=of E] (S) {};
\node[right=of S] (T) {$\mathsmaller{\C D}$};
\node[below=of E] (M) {$\mathsmaller{\C C}$};
\node[below=of T] (N) {$\mathsmaller{\C D}$.};

\draw[->] (E)--(T) node [midway,above] {$\mathsmaller{F}$};
\draw[->] (M)--(N) node [midway,below] {$\mathsmaller{F}$};
\draw[->] (E)--(M) node [midway,left] {$\mathsmaller{\Gamma}$};
\draw[->] (T)--(N) node [midway,left] {$\mathsmaller{\Delta}$};
\draw[MyBlue,->,bend left=20] (E) to node [midway,left] {} (N);
\draw[MyBlue,->,bend right=20] (E) to node [midway,above] {$ \ \ \ \mathsmaller{\stackrel{\eta} \Too}$} (N);

\end{tikzpicture}
\end{center}
 The rule $F_* \colon \Chu(\C C, \Gamma) \to \Chu(\C D, \Delta)$, defined by
 $$(F_*)_0(x; a, f, b) = \big(F_0(x); F_0(a), \eta_x \circ F_1(f) \circ F_{ab}, F_0(b)\big)$$
 \begin{center}
\begin{tikzpicture}

\node (E) at (0,0) {$F_0(a) \times F_0(b)$};
\node[right=of E] (T) {$F_0(a \times b)$};
\node[right=of T] (F) {$F_0(\Gamma_0(x))$};
\node[right=of F] (D) {$\Delta_0(F_0(x))$};

\draw[MyBlue,->] (E)--(T) node [midway,above] {$F_{ab}$};
\draw[MyBlue,->] (T)--(F) node [midway,above] {$F_1(f)$};
\draw[MyBlue,->] (F)--(D) node [midway,above] {$\eta_x$};

\end{tikzpicture}
\end{center}
$$(F_*)_1\big(\phi^0, \phi^+, \phi^-\big) \colon \big(F_0(a), \eta_x \circ F_1(f) \circ F_{ab}, F_0(b)\big)
\to \big(F_0(y); F_0(c), \eta_y \circ F_1(g) \circ F_{cd}, F_0(d)\big),$$
$$(F_*)_1\big(\phi^0, \phi^+, \phi^-\big) = \big(F_1(\phi^0), F_1(\phi^+), F_1(\phi^-)\big),$$
where $\big(\phi^0, \phi^+, \phi^-\big) \colon (x; a, f, b) \to (y; c, g, d)\big)$, 
is a functor.
\end{lemma}

\begin{proof}
 We show that $(F_*)_1\big(\phi^0, \phi^+, \phi^-\big)$ is well-defined i.e., the following diagram commutes:
 \begin{center}
\begin{tikzpicture}

\node (E) at (0,0) {$\mathsmaller{F_0(a) \times F_0(d)} \  $};
\node[right=of E] (T) {};
\node[right=of T] (F) {$ \ \mathsmaller{F_0(a) \times F_0(b)}$};
\node[below=of F] (L) {$\mathsmaller{\Delta_0(F_0(x))}$};
\node[below=of L] (S) {$ \ \  \mathsmaller{\Delta_0(F_0(y))}.$};
\node[left=of S] (K) {};
\node[left=of K] (A) {$\mathsmaller{F_0(c) \times F_0(d)}  \ $};

\draw[->] (E)--(A) node [midway,left] {$\mathsmaller{F_1(\phi^+) \times 1_{F_0(d)}}$};
\draw[->] (F)--(L) node [midway,right] {$\mathsmaller{\eta_x \circ F_1(f) \circ F_{ab}}$};
\draw[->] (E)--(F) node [midway,above] {$\mathsmaller{1_{F_0(a)} \times F_1(\phi^-)}$};
\draw[->] (L)--(S) node [midway,right] {$\mathsmaller{\Delta_1(F_1(\phi^0))}$};
\draw[->] (A)--(S) node [midway,below] {$\mathsmaller{\eta_y \circ F_1(g) \circ F_{cd}}$};

\end{tikzpicture}
\end{center}
Let 
$$A = \Delta_1(F_1(\phi^0)) \circ \eta_x \circ F_1(f) \circ F_{ab} \circ [1_{F_0(a)} \times F_1(\phi^-)],$$
$$B = \eta_y \circ F_1(g) \circ F_{cd} \circ [F_1(\phi^+) \times 1_{F_0(d)}].$$
By the definition of a morphism $\big(\phi^0, \phi^+, \phi^-\big) \colon (x; a, f, b) \to (y; c, g, d)\big)$ we get 
$$\Gamma_1(\phi^0) \circ f \circ (1_a \times \phi^- = g \circ (\phi^+ \times 1_d) \To$$
$$(\ast) \ \ \ \ \ F_1(\Gamma_1(\phi^0)) \circ F_1(f) \circ F_1(1_a \times \phi^-) = F_1(g) \circ F_1(\phi^+ \times 1_d).$$
As $ F_1(\phi^+ \times 1_d) \circ F_{ad} =   F_{cd} \circ [F_1(\phi^+) \times F_1(1_d)]$,
and since the following
rectangle commutes
\begin{center}
\begin{tikzpicture}

\node (E) at (0,0) {$\mathsmaller{F_0(\Gamma_0(x))} $};
\node[right=of E] (T) {};
\node[right=of T] (F) {$\mathsmaller{F_0(\Gamma_0(y))}$};
\node[below=of F] (L) {$\mathsmaller{\Delta_0(F_0(y))}$,};
\node[below=of E] (M) {$\mathsmaller{\Delta_0(F_0(x))}$};

\draw[->] (E)--(F) node [midway,above] {$\mathsmaller{F_1(\Gamma_1(\phi^0))}$};
\draw[->] (F)--(L) node [midway,right] {$\mathsmaller{\eta_y}$};
\draw[->] (E)--(M) node [midway,left] {$\mathsmaller{\eta_x}$};
\draw[->] (M)--(L) node [midway,below] {$\mathsmaller{\Delta_1(F_1(\phi^0))}$};

\end{tikzpicture}
\end{center}
\begin{align*}
 A & = \Delta_1(F_1(\phi^0)) \circ \eta_x \circ F_1(f) \circ F_{ab} \circ [F_1(1_a) \times F_1(\phi^-)]\\
 & =  \Delta_1(F_1(\phi^0)) \circ \eta_x \circ F_1(f) \circ F_1(1_a \times \phi^-) \circ F_{ad}\\
 & = \eta_y \circ F_1(\Gamma_1(\phi^0)) \circ F_1(f) \circ F_1(1_a \times \phi^-) \circ F_{ad}\\
 & \stackrel{(\ast)} =  \eta_y \circ F_1(g) \circ F_1(\phi^+ \times 1_d) \circ F_{ad}\\ 
 & = \eta_y \circ F_1(g) \circ F_{cd} \circ [F_1(\phi^+) \times F_1(1_d)]\\
 & = B.
\end{align*}
The preservation of units and compositions by $F_*$ is immediate to show.
\end{proof}

Next we define the appropriate category on which the generalised global Chu functor will be defined. Notice that this category
is not a special case of the Grothendieck construction, but a variation of it. 

\begin{definition}[The category of pairs of ccc's and endofunctors]\label{def: G}
Let the category 
$$\sum_{\C C \in \ccCat}\End(\C C)$$
with objects pairs $(\C C, \Gamma)$, where $\C C$ in $\ccCat$ and $\Gamma \colon \C C \to \C C$ an 
endofunctor on $\C C$, and
morphisms $(F, \eta) \colon (\C C, \Gamma) \to (\C D, \Delta)$, where $F \colon \C C \to \C D$ is a product preserving 
functor and $\eta \colon F \circ \Gamma \To \Delta \circ F$. If $(G, \theta) \colon (\C D, \Delta) \to (\C E, E)$,
let $(G, \theta) \circ (F, \eta) \colon (\C C, \Gamma) \to (\C E, E)$ be defined by
\begin{center}
\begin{tikzpicture}

\node (E) at (0,0) {$\mathsmaller{\C C} $};
\node[right=of E] (S) {};
\node[right=of S] (T) {$\mathsmaller{\C D}$};
\node[right=of T] (X) {};
\node[right=of X] (F) {$\mathsmaller{\C E}$};
\node[below=of F] (L) {$\mathsmaller{\C E}$};
\node[below=of E] (M) {$\mathsmaller{\C C}$};
\node[below=of T] (N) {$\mathsmaller{\C D}$};

\draw[->] (E)--(T) node [midway,above] {$\mathsmaller{F}$};
\draw[->] (T)--(F) node [midway,above] {$\mathsmaller{G}$};
\draw[->] (M)--(N) node [midway,below] {$\mathsmaller{F}$};
\draw[->] (N)--(L) node [midway,below] {$\mathsmaller{G}$};
\draw[->] (E)--(M) node [midway,left] {$\mathsmaller{\Gamma}$};
\draw[->] (T)--(N) node [midway,left] {$\mathsmaller{\Delta}$};
\draw[->] (F)--(L) node [midway,right] {$\mathsmaller{E}$};
\draw[MyBlue,->,bend left=20] (E) to node [midway,left] {} (N);
\draw[MyBlue,->,bend right=20] (E) to node [midway,above] {$ \ \ \ \mathsmaller{\stackrel{\eta} \Too}$} (N);
\draw[MyBlue,->,bend left=20] (T) to node [midway,left] {} (L);
\draw[MyBlue,->,bend right=20] (T) to node [midway,above] {$ \ \ \ \mathsmaller{\stackrel{\theta} \Too}$} (L);

\end{tikzpicture}
\end{center}
$$(G, \theta) \circ (F, \eta) = (G \circ F, \theta \ast \eta),$$
$$\theta \ast \eta \colon (G \circ F) \circ \Gamma \To E \circ (G \circ F),$$
$$(\theta \ast \eta)_a \colon G_0(F_0(\Gamma_0(a))) \to E_0(G_0(F_0(a))),$$
$$(\theta \ast \eta)_a = \theta_{F_0(a)} \circ G_1(\eta_a)$$
\begin{center}
\begin{tikzpicture}

\node (E) at (0,0) {$\mathsmaller{G_0(F_0(\Gamma_0(a)))} $};
\node[right=of E] (S) {};
\node[right=of S] (T) {$\mathsmaller{G_0(\Delta_0(F_0(a)))}$};
\node[below=of T] (U) {$\mathsmaller{E_0(G_0(F_0(a)))}$};

\draw[MyBlue,->] (E)--(T) node [midway,above] {$\mathsmaller{G_1(\eta_a)}$};
\draw[MyBlue,->] (T)--(U) node [midway,right] {$\mathsmaller{\theta_{F_0(a)}}$};
\draw[->] (E)--(U) node [midway,left] {$\mathsmaller{(\theta \ast \eta)_a \ \ }$};

\end{tikzpicture}
\end{center}
Moreover, $1_{(\C C, \Gamma)} = \big(\Id^{\C C}, 1_{\Gamma}\big)$.
\end{definition}

First we explain why $\theta \ast \eta$ is a natural transformation $(G \circ F) \circ \Gamma \To E \circ (G \circ F)$.
If $f \colon a \to b$ in $C_1$, then, as $\eta \colon F \circ \Gamma \To \Delta \circ F$, the following left
rectangle commutes:
\begin{center}
\begin{tikzpicture}

\node (E) at (0,0) {$\mathsmaller{F_0(\Gamma_0(a))} $};
\node[right=of E] (T) {};
\node[right=of T] (F) {$\mathsmaller{F_0(\Gamma_0(b))}$};
\node[below=of F] (L) {$\mathsmaller{\Delta_0(F_0(b))}$};
\node[below=of E] (M) {$\mathsmaller{\Delta_0(F_0(a))}$};

\node[right=of F] (A) {$\mathsmaller{G_0(\Delta_0(F_0(a)))}$};
\node[right=of A] (B) {};
\node[right=of B] (C) {$\mathsmaller{G_0(\Delta_0(F_0(b)))}$};
\node[below=of C] (D) {$\mathsmaller{E_0(G_0(F_0(b)))}$};
\node[below=of A] (X) {$\mathsmaller{E_0(G_0(F_0(a)))}$};

\draw[->] (E)--(F) node [midway,above] {$\mathsmaller{F_1(\Gamma_1(f))}$};
\draw[->] (F)--(L) node [midway,right] {$\mathsmaller{\eta_b}$};
\draw[->] (E)--(M) node [midway,right] {$\mathsmaller{\ \ \ \ \ \ \ \ \ \ \ \ \#_1} $};
\draw[dashed,->] (E)--(M) node [midway,left] {$\mathsmaller{\eta_a}$};
\draw[->] (M)--(L) node [midway,below] {$\mathsmaller{\Delta_1(F_1(f))}$};

\draw[->] (A)--(C) node [midway,above] {$\mathsmaller{G_1(\Delta_1(F_1(f)))}$};
\draw[->] (C)--(D) node [midway,right] {$\mathsmaller{\theta_{F_0(b)}}$};
\draw[->] (A)--(X) node [midway,right] {$\mathsmaller{\ \ \ \ \ \ \ \ \ \ \ \ \ \ \ \#_2} $};
\draw[->] (A)--(X) node [midway,left] {$\mathsmaller{\theta_{F_0(a)}}$};
\draw[->] (X)--(D) node [midway,below] {$\mathsmaller{E_1(G_1(F_1(f)))}$};

\end{tikzpicture}
\end{center}
By commutativity $(\#_1)$ we get
$$(\ast) \ \ \ \ \ \ \  G_1(\eta_b) \circ G_1(F_1(\Gamma_1(f))) = G_1(\Delta_1(F_1(f))) \circ G_1(\eta_a).$$
As $\theta \colon G \circ \Delta \To E \circ G$, and $F_1(f) \colon F_0(a) \to F_0(b)$ in $D_1$, the above right rectangle 
commutes. The commutativity of the following rectangle diagram follows:
\begin{center}
\begin{tikzpicture}

\node (E) at (0,0) {$\mathsmaller{G_0(F_0(\Gamma_0(a)))} $};
\node[right=of E] (T) {};
\node[right=of T] (F) {$\mathsmaller{G_0(F_0(\Gamma_0(b)))}$};
\node[below=of F] (L) {$\mathsmaller{E_0(G_0(F_0(b)))}$,};
\node[below=of E] (M) {$\mathsmaller{E_0(G_0(F_0(a)))}$};

\draw[->] (E)--(F) node [midway,above] {$\mathsmaller{G_1(F_1(\Gamma_1(f)))}$};
\draw[->] (F)--(L) node [midway,right] {$\mathsmaller{(\theta \ast \eta)_b}$};
\draw[->] (E)--(M) node [midway,left] {$\mathsmaller{(\theta \ast \eta)_a}$};
\draw[->] (M)--(L) node [midway,below] {$\mathsmaller{E_1(G_1(F_1(f)))}$};

\end{tikzpicture}
\end{center}
\begin{align*}
 (\theta \ast \eta)_b \circ G_1(F_1(\Gamma_1(f))) & = \theta_{F_0(b)} \circ G_1(\eta_b) \circ G_1(F_1(\Gamma_1(f)))\\
 & \stackrel{(\ast)} = \theta_{F_0(b)} \circ G_1(\Delta_1(F_1(f))) \circ G_1(\eta_a)\\
 & \stackrel{(\#_2)} = E_1(G_1(F_1(f))) \circ \theta_{F_0(a)} \circ G_1(\eta_a)\\
 & =  E_1(G_1(F_1(f))) \circ (\theta \ast \eta)_a.
\end{align*}
If $(F, \eta) \colon (\C C, \Gamma) \to (\C D, \Delta)$, then 
$(F, \eta) \circ 1_{(\C C, \Gamma)} = (F, \eta) \circ  \big(\Id^{\C C}, 1_{\Gamma}\big) = 
(F \circ \Id^{\C C}, \eta \ast 1_G) = (F, \eta)$, as $\eta \ast 1_G = \eta$. Similarly, if $1_{(\C D, \Delta)} \circ 
(F, \eta) = \big(\Id^{\C D}, 1_{\Delta}\big) = (\Id^{\C D} \circ F, 1_{\Delta} \ast \eta) = (F, \eta)$, as 
$1_{\Delta} \ast \eta = \eta$. If $(G, \theta) \colon (\C D, \Delta) \to (\C E, E)$ and 
$(H, \rho) \colon (\C E, E) \to (\C Z, Z)$, then 
$$(H, \rho) \circ [(G, \theta) \circ (F, \eta)] = \big(H \circ (G \circ F), \rho \ast (\theta \ast \eta)\big),$$
$$[(H, \rho) \circ (G, \theta)] \circ (F, \eta) = \big((H \circ G) \circ F, (\rho \ast \theta) \ast \eta\big),$$
and as $\rho \ast (\theta \ast \eta) = (\rho \ast \theta) \ast \eta$, we get 
$(H, \rho) \circ [(G, \theta) \circ (F, \eta)] = [(H, \rho) \circ (G, \theta)] \circ (F, \eta)$.

\begin{theorem}[The generalised global Chu functor]\label{prp: Gextfunctor}
The rule 
$$\CHU \colon \sum_{\C C \in \ccCat}\End(\C C) \to \Cat,$$
$$\CHU_0(\C C, \Gamma) = \Chu(\C C, \Gamma),$$
$$\CHU_1\big(F, \eta) \colon (\C C, \Gamma) \to (\C D, \Delta)\big) \colon \Chu(\C C, \Gamma) \to \Chu(\C D, \Delta),$$
$$\CHU_1\big(F, \eta) = F_*,$$
where $F_*$ is defined in Lemma~\ref{lem: Lemextfunctor}, is a functor. Moreover, 
if $F \colon \C C \to \C D$ is a full embedding and
$\eta_a$ is a monomorphism, for every $a \in C_0$, then $F_*$ is a full embedding of $\Chu(\C C, \Gamma)$
into $\Chu(\C D, \Delta)$. 
\end{theorem}

\begin{proof}
 By Lemma~\ref{lem: Lemextfunctor} $\CHU_1(F, \phi)$ is well-defined. Clearly,
 $$\CHU_1(1_{(\C C, \Gamma)}) = \CHU_1 \big(\Id^{\C C}, 1_{\Gamma}\big) = \big[\Id^{\C C}\big]_* = 1_{\CHU(\C C, \Gamma)}.$$
 If $(G, \theta) \colon (\C D, \Delta) \to (\C E, E)$, we show that 
 $\CHU_1(G \circ F, \theta \ast \eta) = (G \circ F)_* = G_* \circ F_* = \CHU_1(G, \theta) \circ \CHU_1(F, \eta)$.
 If $A = \big[(G \circ F)_*\big]_0(x; a, f, b)$ and $B = (G_*)_0\big((F_*)_0(x; a, f, b)\big)$, then
 \begin{align*}
  A & = \big(G_0(F_0(x)); G_0(F_0(a)), (\theta \circ \eta)_x \circ G_1(F_1(f))
  \circ (G \circ F)_{ab}, G_0(F_0(b))\big)\\
  & = \big(G_0(F_0(x)); G_0(F_0(a)), \theta_{F_0(x)} \circ G_1(\eta_x) \circ G_1(F_1(f)) \circ G_1(F_{ab}) \circ G_{F_0(a)F_0(b)}, 
  G_0(F_0(b))\big)\\
  & = \big(G_0(F_0(x)); G_0(F_0(a)), \theta_{F_0(x)} \circ G_1\big[\eta_x \circ F_1(f) \circ F_{ab}\big] \circ 
  G_{F_0(a)F_0(b)},  G_0(F_0(b))\big)\\
  & = (G_*)_0\big(F_0(x); F_0(a), \eta_x \circ F_1(f) \circ F_{ab}, F_0(b)\big)\\
  & = B.
 \end{align*}
The equality $[(G \circ F)_*]_1(\phi^0, \phi^+, \phi^-) = (G_*)_1\big((F_*)_1(\phi^0, \phi^+, \phi^-)\big)$
follows immediately.
Let $F \colon \C C \to \C D$ be a full embedding and
$\eta_a$ a monomorphism, for every $a \in C_0$. The equality $\big(F_0(x); F_0(a), \eta_x \circ F_1(f) \circ F_{ab}, 
F_0(b)\big) =  \big(F_0(x{'}); F_0(a{'}), \eta_{x{'}} \circ F_1(f{'}) \circ F_{a{'}b{'}}, F_0(b{'})\big)$ 
implies $x = x{'}, a = a{'}, b = b{'}$, and as 
 $\eta_x$ is a monomorphism and $F_ab$ an isomorphism, hence an epimorphism, we get $ F_1(f) =  F_1(f{'})$, hence
 $f = f{'}$. The fact that $F_*$ is faithful and full follows immediately.
\end{proof}

The local generalised Chu functor is a special case of the global one. Namely, if $\C D = \C C$, $F = \Id^{\C C}$, 
and $\Gamma, \Delta \colon \C C \to \C C$, and if $\eta \colon \Id^{\C C} \circ \Gamma \To \Delta \circ \Id^{\C C}$ i.e., 
$\eta \colon \Gamma \To \Delta$, then 
$$\Chu_1^{\C C}(\eta) = \big[\Id^{\C C}\big]_*.$$

%
%
%
%
%
%
%
%
%
%
%
%
%
%
%
%
%
%
%
%

\section{A generalised Chu representation of the category of predicates}
\label{sec: predchu}

Predicates on sets were organised in a category that was called $\Pred$ in~\cite{Ja99}, in order to describe 
the logic and type theory of standard sets in fibred form. Here we present this category within 
$\BST$.

\begin{definition}\label{def: predcat} 
The objects of the category of predicates $\Pred$ are triplets 
$(X, i_A^X, A)$, where $X$ is a set and $(A, i_A^X)$ is a subset of $X$. If $(X,i_A^X, A)$ and $(Y, i_B^Y, B)$ are objects
of $\Pred$, a morphism $u \colon (X, i_A^X, A) \to (Y, i_B^Y, B)$ in $\Pred$ is a pair of functions
$u = \big(u^0, u^+\big)$, where
$u^0 \colon X \to Y$ and $u^+ \colon A \to B$ such that the following diagram commutes
\begin{center}
\begin{tikzpicture}

\node (E) at (0,0) {$X$};
\node[right=of E] (T) {};
\node[right=of T] (F) {$Y$.};
\node[above=of E] (A) {$A$};
\node[above=of F] (B) {$B$};

\draw[right hook->] (A)--(E) node [midway,left] {$i_{A}^X$};
\draw[left hook->] (B)--(F) node [midway,right] {$i_{B}^Y$};
\draw[->] (A)--(B) node [midway,above] {$u^+$};
\draw[->] (E)--(F) node [midway,below] {$u^0$};

\end{tikzpicture}
\end{center}
If $v = \big(v^0, v^+\big) \colon (Y, i_B^Y, B) \to (Z, i_C^Z, C)$, let $v \circ u \colon (X, i_A^X, A) \to (Z, i_C^Z, C)$, defined
by $v \circ u = \big(v^0 \circ u^0, v^+ \circ u^+\big)$. Moreover, $1_{(X, i_A^X, A)} = \big(\id_X, \id_A\big)$.
\end{definition}

In~\cite{Ja99}, p.~11, the embedding 
$i_A^X \colon A \to X$ is omitted for simplicity, and 
a morphism $u$ is just a function $u^0 \colon X \to Y$ such that
\[\forall_{a \in A}\exists_{b \in B}\big(u^0(i_A^X(a)) =_Y i_B^Y(b)\big).\]
It is immediate to see that to each $a \in A$ there is a unique (up to the equality of $B$) $b \in B$ such that 
$u^0(i_A^X(a)) =_Y i_B^Y(b)$. By Myhill's principle of non-choice (or unique choice), introduced
in~\cite{My75}, there is a (necessarily)
unique map $u^+$ that makes the above diagram commutative. As this principle is avoided in $\BST$, we
prefer to present a morphism $u$ in $\Pred$ as a pair $(u^0, u^+)$. 
It is immediate to see that if $u^0$ is an embedding, then $u^+$ is an embedding, and 
if $u^0$ is strongly extensional, then $u^+$ is also strongly extensional. For a specific set $X$ the ``fibre'' category $\Pred_X$ is the subcategory of $\Pred$ with objects triplets of the form
$(X, A, i_A^X)$ with $X$ fixed, while a morphism $u \colon (X, A, i_A^X) \to (X, B, i_B^X)$ is a pair 
$(\id_X, u_{\mathsmaller{AB}})$, and the required commutativity of the following diagram
\begin{center}
\begin{tikzpicture}

\node (E) at (0,0) {$X$};
\node[right=of E] (T) {};
\node[right=of T] (F) {$X$};
\node[above=of E] (A) {$A$};
\node[above=of F] (B) {$B$};

\draw[right hook->] (A)--(E) node [midway,left] {$i_{A}^X$};
\draw[left hook->] (B)--(F) node [midway,right] {$i_{B}^X$};
\draw[->] (A)--(B) node [midway,above] {$u_{\mathsmaller{AB}}$};
\draw[->] (E)--(F) node [midway,below] {$\id_X$};

\end{tikzpicture}
\end{center}
expresses that $u_{\mathsmaller{AB}} \colon A \subseteq B$. Hence $\Pred_X$ is identified with the category $\B {\C P}(X)$.

\begin{proposition}[Generalised Chu representations of $\Set$ and $\Pred$]\label{prp: bishopchu3}
\normalfont (i)
\itshape 
The functor
 $E^{\Set} \colon \Set \to \Chu(\Set, \Id)$, defined by
 $$E_0^{\Set}(X) = \big(X; X, I_{X}^X, \D 1\big),$$
 $$I_X^X \colon X \times \D 1 \to \Id_0(X) = X, \ \ \ I_X^X(x, 0) = x; \ \ \ x \in X,$$
 $$E_1^{\Set}\big(f \colon X \to Y\big) = (f, f, \id_{\D 1}) \colon \big(X, I_{A}^X, A\big) \to \big(Y, i_{B}^Y, B\big)\big) = 
 \big(u^0, u^+, \id_{\D 1}\big) \colon  \big(X; X, I_{X}^X, \D 1\big) \to \big(Y; Y, I_{Y}^Y, \D 1\big),$$
 is a  strict generalised Chu representation of $\Set$ into $\Chu(\Set, \Id)$. \\[1mm]
\normalfont (ii)
\itshape The functor
 $E^{\Pred} \colon \Pred \to \Chu(\Set, \Id)$, defined by
 $$E_0^{\Pred}\big(X, i_{A}^X, A\big) = \big(X; A, I_{A}^X, \D 1\big),$$
 $$I_A^X \colon A \times \D 1 \to \Id_0(X) = X, \ \ \ I_A^X(a, 0) = i_{A}^X(a); \ \ \ a \in A,$$
 $$E_1^{\Pred}\big(u = \big(u^0, u^+\big) \colon \big(X, I_{A}^X, A\big) \to \big(Y, I_{B}^Y, B\big)\big) = 
 \big(u^0, u^+, \id_{\D 1}\big) \colon  \big(X; A, I_{A}^X, \D 1\big) \to \big(Y; B, I_{B}^Y, \D 1\big),$$
 is a strict generalised Chu representation of $\Pred$ into $\Chu(\Set, \Id)$. \\[1mm]
\normalfont (iii)
\itshape If $F \colon \Set \to \Pred$ is the full embedding of 
 $\Set$ into $\Pred$, defined by $F_0(X) = (X, \id_X, X)$ and $F_i(f \colon X \to Y) = (f, f)$, 
 the following diagram commutes
 \begin{center}
\begin{tikzpicture}

\node (E) at (0,0) {$\Pred$};
\node[right=of E] (T) {};
\node[right=of T] (F) {$\mathsmaller{\Chu(\Set, \Id)}$.};
\node[above=of E] (A) {$\Set$};
\node[above=of F] (B) {$\mathsmaller{\Chu(\Set, \Id)}$};

\draw[right hook->] (A)--(E) node [midway,left] {$F$};
\draw[left hook->] (B)--(F) node [midway,right] {$\Id$};
\draw[right hook->] (A)--(B) node [midway,above] {$E^{\Set}$};
\draw[right hook->] (E)--(F) node [midway,below] {$E^{\Pred}$};

\end{tikzpicture}
\end{center}
\end{proposition}

\begin{proof}
We show only (ii). If $u = \big(u^0, u^+\big) \colon \big(X, i_{A}^X, A\big) \to \big(Y, i_{B}^Y, B\big)$, then
$\big(u^0, u^+, \id_{\D 1}\big) \colon  \big(X; A, I_{A}^X, \D 1\big) \to \big(Y; B, I_{B}^Y, \D 1\big)$, as
the commutativity of the rectangle
\begin{center}
\begin{tikzpicture}

\node (E) at (0,0) {$X$};
\node[right=of E] (T) {};
\node[right=of T] (F) {$Y$};
\node[above=of E] (A) {$A$};
\node[above=of F] (B) {$B$};

\draw[right hook->] (A)--(E) node [midway,left] {$i_{A}^X$};
\draw[left hook->] (B)--(F) node [midway,right] {$i_{B}^Y$};
\draw[->] (A)--(B) node [midway,above] {$u^+$};
\draw[->] (E)--(F) node [midway,below] {$u^0$};

\end{tikzpicture}
\end{center}
implies the commutativity of the following diagram 
\begin{center}
\begin{tikzpicture}

\node (E) at (0,0) {$  \mathsmaller{A \times \D 1}  $};
\node[right=of E] (T) {};
\node[right=of T] (F) {$ \mathsmaller{A \times \D 1}$};
\node[below=of F] (L) {$\mathsmaller{X}$};
\node[below=of L] (S) {$ \ \mathsmaller{Y}$};
\node[left=of S] (K) { };
\node[left=of K] (A) {$\mathsmaller{B \times \D 1} \ \ $};

\draw[->] (E)--(A) node [midway,left] {$\mathsmaller{u^+ \times \id_{\D 1}}$};
\draw[->] (F)--(L) node [midway,right] {$\mathsmaller{I_A^X}$};
\draw[->] (E)--(F) node [midway,above] {$\mathsmaller{\id_A \times \id_{\D 1}}$};
\draw[->] (L)--(S) node [midway,right] {$\mathsmaller{\Id_1(u^0) = u^0}$};
\draw[->] (A)--(S) node [midway,below] {$\mathsmaller{I_B^Y}$};

\end{tikzpicture}
\end{center}
$$u^0\big(I_A^X(a, 0)\big) = u^0\big(i_A^X(a)\big) = i_B^Y\big(u^+(a)\big) = I_B^Y\big(u^+(a), 0\big).$$
Clearly, $E^{\Pred}$ is injective on objects and arrows, hence $E^{\Pred}$ is an embedding. 
It is also full, as if
$\big(u^0, u^+, \id_{\D 1}\big) \colon  \big(X; A, I_{A}^X, \D 1\big) \to \big(Y; B, I_{B}^Y, \D 1\big)$, then 
$u = \big(u^0, u^+\big) \colon \big(X, i_{A}^X, A\big) \to \big(Y, i_{B}^Y, B\big)$, because the commutativity of the 
last diagram implies the commutativity of the first rectangle. 
\end{proof}

\begin{definition}\label{def: predC}
 If $\C C$ is a category, the category $\Pred(\C C)$ of $\C C$ has objects pairs $(x, i \colon a \eto x)$, where $x \in C_0$
 and $i \in C_1(a, x)$ is a monomorphism, and morphisms $(f^0, f^+) \colon (x, i \colon a \eto x) \to (y, j \colon b \eto y)$
 with $j \circ f^+ = f_0 \circ i$
\begin{center}
\begin{tikzpicture}

\node (E) at (0,0) {$x$};
\node[right=of E] (T) {};
\node[right=of T] (F) {$y$.};
\node[above=of E] (A) {$a$};
\node[right=of A] (C) {};
\node[right=of C] (B) {$ \ b$};

\draw[right hook->] (A)--(E) node [midway,left] {$i$};
\draw[left hook->] (B)--(F) node [midway,right] {$j$};
\draw[->] (A)--(B) node [midway,above] {$f^+$};
\draw[->] (E)--(F) node [midway,below] {$f^0$};

\end{tikzpicture}
\end{center}
If $(g^0, g^+) \colon (y, j \colon b \eto y) \to (z, k \colon e \eto z)$, then $(g^0, g^+) \circ (f^0, f^+) = 
(g^0 \circ f^0, g^+ \circ f^+)$. Moreover, $1_{(x, i \colon a \eto x)} = (1_x, 1_a)$. 
\end{definition}

\begin{proposition}[Generalised Chu representation of $\Pred(\C C)$]\label{prp: predCrepr}
If $\C C$ is a $\CCC$, the functor
 $$E^{\Pred(\C C)} \colon \Pred(\C C) \to \Chu(\C C, \Id^{\C C}),$$
 $$E_0^{\Pred(\C C)}\big(x, i \colon a \eto x\big) = \big(x; a, i \circ \pr_a, 1\big),$$
 \begin{center}
\begin{tikzpicture}

\node (E) at (0,0) {$a \times 1$};
\node[right=of E] (T) {$a$};
\node[right=of T] (F) {$x$};

\draw[right hook->] (E)--(T) node [midway,above] {$\pr_a$};
\draw[right hook->] (T)--(F) node [midway,above] {$i$};

\end{tikzpicture}
\end{center}
$$E_1^{\Pred(\C C)}\big(\big(f^0, f^+\big) \colon \big(x, i \colon a \eto x\big) \to \big(y, j \colon b \eto y\big)\big) = 
 \big(f^0, f^+, 1_{1}\big) \colon  \big(x; a, i \circ \pr_a, 1\big) \to \big(y; b, j \circ \pr_b, 1\big),$$
is a strict generalised Chu representation of $\Pred(\C C)$ into $\Chu(\C C, \Id^{\C C})$.
\end{proposition}

\begin{proof}
The morphism $\pr_a$ is an iso, hence a mono. To show that $E_1^{\Pred(\C C)}\big(f^0, f^+\big)
\colon  \big(x; a, i \circ \pr_a, 1\big) \to \big(y; b, j \circ \pr_b, 1\big)$, we show that the following diagram commutes
\begin{center}
\begin{tikzpicture}

\node (E) at (0,0) {$  \mathsmaller{a \times 1}  $};
\node[right=of E] (T) {};
\node[right=of T] (F) {$ \mathsmaller{a \times 1}$};
\node[below=of F] (L) {$\mathsmaller{x}$};
\node[below=of L] (S) {$ \ \mathsmaller{y}$};
\node[left=of S] (K) { };
\node[left=of K] (A) {$\mathsmaller{b \times 1} \ \ $};

\draw[->] (E)--(A) node [midway,left] {$\mathsmaller{f^+ \times 1_{1}}$};
\draw[->] (F)--(L) node [midway,right] {$\mathsmaller{i \circ \pr_a}$};
\draw[->] (E)--(F) node [midway,above] {$\mathsmaller{1_{a \times 1}}$};
\draw[->] (L)--(S) node [midway,right] {$\mathsmaller{f^0}$};
\draw[->] (A)--(S) node [midway,below] {$\mathsmaller{j \circ \pr_b}$};

\end{tikzpicture}
\end{center}
\begin{align*}
 f^0 \circ (i \circ \pr_a) & = (f^0 \circ i) \circ \pr_a\\
 & = (j \circ f^+) \circ \pr_a\\
 & = j \circ (f^+ \circ \pr_a)\\
 & = j \circ [\pr_b \circ (f^+ \times 1_1)]\\
 & = (j \circ \pr_b) \circ (f^+ \times 1_1)
\end{align*}
as the equality $f^+ \circ \pr_a = \pr_b \circ (f^+ \times 1_1)$ follows as in the proof of 
Proposition~\ref{prp: subrepr}.
%
%
%
If $\big(x; a, i \circ \pr_a, 1\big) = \big(y; b, j \circ \pr_b, 1\big)$, then $x = y$, $a = b$, and $i \circ \pr_a = j \circ p_a$.
As $\pr_a$ is a mono, we get $i = j$, and hence $E^{\Pred(\C C)}$ is injective on objects. It is trivially 
injective on arrows. To show that it is full, let $(\phi^0, \phi^+, \phi^-) \colon  \big(x; a, i \circ \pr_a, 1\big) \to 
\big(y; b, j \circ \pr_b, 1\big)$. Clearly, $\phi^- = 1_1$. Moreover, by the previous equalities we get
$(\phi^0 \circ i) \circ \pr_a = (j \circ \phi^+) 
\circ \pr_a$, and since $\pr_a$ is a mono, we conclude that $\phi^0 \circ i = j \circ \phi^+$ i.e., $(\phi^0, \phi^+) \colon 
 \big(x, i \colon a \eto x\big) \to \big(y, j \colon b \eto y\big)$.
\end{proof}

%
%
%
%
%
%

\section{A generalised Chu representation of the category of complemented predicates}
\label{sec: complpredchu}

Here we organise the complemented predicates on sets that are equipped with a fixed inequality
in a category $\Pred^{\neq}$. Its
subcategory $\Pred^{\neq}_{\se}$ is formed by considering in the definition of the morphisms in $\Pred^{\neq}$
strongly extensional functions. The motivation behind the next definition is to get a strict generalised Chu representation of  
$\Pred^{\neq}(\Set)$ into the Chu category over $\Set$ and the endofunctor $\Id^2 \colon \Set \to \Set$, defined by
$$\Id^2_0(X) = X \times X,$$
$$\Id^2_1(f \colon X \to Y) : X \times X \to Y \times Y,$$
$$[\Id^2_1(f)](x, x{'}) = \big(f(x), f(x{'})\big).$$
This result is in complete analogy to the full embedding of $\Pred$ into $\Chu(\Set, \Id)$.

\begin{definition}\label{def: predineq}
The category $\Predc(\Set)$ of complemented predicates has objects pairs $(X, \B A)$, where $X$ is in $\Setc$,
the category of sets equipped with a fixed inequality and strongly extensional functions between them,
and $\B A := (A^1, A^0)$ is a complemented subset of $X$. If $(X, \B A)$ and $(Y, \B B)$ are objects of $\Predc$, a
morphism $u \colon (X, \B A) \to (Y, \B B)$ is a triplet $u = \big(u^0, u^+, u^-\big)$, where $u^0 \colon X \to Y$, 
$u^+ \colon A^1 \to B^1$, and $u^- \colon B^0 \to A^0$ such that the following rectangles commute
\begin{center}
\begin{tikzpicture}

\node (E) at (0,0) {$X$};
\node[right=of E] (T) {};
\node[right=of T] (F) {$Y$};
\node[above=of E] (A) {$A^1$};
\node[above=of F] (B) {$B^1$};

\node[right=of F] (X) {};
\node[right=of X] (K) {$Y$};
\node[right=of K] (S) {};
\node[right=of S] (L) {$X$.};
\node[above=of K] (M) {$B^0$};
\node[above=of L] (N) {$A^0$};

\draw[right hook->] (A)--(E) node [midway,left] {$i_{A^1}^X$};
\draw[left hook->] (B)--(F) node [midway,right] {$i_{B^1}^Y$};
\draw[->] (A)--(B) node [midway,above] {$u^+$};
\draw[->] (E)--(F) node [midway,below] {$u^0$};

\draw[right hook->] (M)--(K) node [midway,left] {$i_{B^0}^Y$};
\draw[left hook->] (N)--(L) node [midway,right] {$i_{A^0}^X$};
\draw[->] (M)--(N) node [midway,above] {$u^-$};
\draw[->] (L)--(K) node [midway,below] {$u^0$};

\end{tikzpicture}
\end{center}
If $u = \big(v^0, v^+, v^-\big) \colon (Y, \B B) \to (Z, \B C)$, we define the composite morphism 
$v \circ u \colon (X, \B A) \to (Z, \B C)$ by
$v \circ u = \big(v^0 \circ u^0, v^+ \circ u^+, u^- \circ v^-\big)$. Moreover, $1_{(X, \B A)} = \big(\id_X, \id_{A^1}, 
\id_{A^0}\big)$.

\end{definition}

\begin{proposition}[Generalised Chu representation of $\Pred^{\neq}(\Set)$]\label{prp: bishopchu4}
 The functor
 $$E^{\Pred^{\neq}(\Set)} \colon \Pred^{\neq} \to \Chu(\Set, \Id^2),$$
 $$E_0^{\Pred^{\neq}(\Set)}\big(X, \B A\big) = \big(X; A^1, i_{A^1}^X \times i_{A^0}^X, A^0\big),$$
 $$i_{A^1}^X \times i_{A^0}^X \colon A^1 \times A^0 \to \Id^2_0(X) = X \times X,$$
 $$E_1^{\Pred^{\neq}(\Set)}\big(u^0, u^+, u^-\big) = 
 \big(u^0, u^+, u^-\big) \colon  \big(X; A^1, i_{A^1}^X \times i_{A^0}^X, A^0\big) \to 
 \big(Y; B^1, i_{B^1}^Y \times i_{B^0}^Y, B^0\big),$$
 where $\big(u^0, u^+, u^-\big) \colon \big(X, \B A\big) \to \big(Y, \B B\big)$, is a strict generalised Chu representation of $\Pred^{\neq}(\Set)$ into $\Chu(\Set, \Id^2)$.
\end{proposition}

\begin{proof}
If $\big(u^0, u^+, u^-\big) \colon \big(X, \B A\big) \to \big(Y, \B B\big)$, then
$\big(u^0, u^+, u^-\big) \colon  \big(X; A^1, i_{A^1}^X \times i_{A^0}^X, A^0\big) \to 
 \big(Y; B^1, i_{B^1}^Y \times i_{B^0}^Y, B^0\big)$, as
the commutativity of the following two rectangles
\begin{center}
\begin{tikzpicture}

\node (E) at (0,0) {$X$};
\node[right=of E] (T) {};
\node[right=of T] (F) {$Y$};
\node[above=of E] (A) {$A^1$};
\node[above=of F] (B) {$B^1$};

\node[right=of F] (X) {};
\node[right=of X] (K) {$Y$};
\node[right=of K] (S) {};
\node[right=of S] (L) {$X$.};
\node[above=of K] (M) {$B^0$};
\node[above=of L] (N) {$A^0$};

\draw[right hook->] (A)--(E) node [midway,left] {$i_{A^1}^X$};
\draw[left hook->] (B)--(F) node [midway,right] {$i_{B^1}^Y$};
\draw[->] (A)--(B) node [midway,above] {$u^+$};
\draw[->] (E)--(F) node [midway,below] {$u^0$};

\draw[right hook->] (M)--(K) node [midway,left] {$i_{B^0}^Y$};
\draw[left hook->] (N)--(L) node [midway,right] {$i_{A^0}^X $};
\draw[->] (M)--(N) node [midway,above] {$u^-$};
\draw[->] (L)--(K) node [midway,below] {$u^0$};

\end{tikzpicture}
\end{center}
implies the commutativity of the following diagram 
\begin{center}
\begin{tikzpicture}

\node (E) at (0,0) {$  \mathsmaller{A^1 \times B^0} \ $};
\node[right=of E] (T) {};
\node[right=of T] (F) {$ \mathsmaller{A^1 \times A^0}$};
\node[below=of F] (L) {$ \ \mathsmaller{X \times X}$};
\node[below=of L] (S) {$ \ \ \mathsmaller{Y \times Y}$};
\node[left=of S] (K) { };
\node[left=of K] (A) {$\mathsmaller{B^1 \times B^0} $};

\draw[->] (E)--(A) node [midway,left] {$\mathsmaller{u^+ \times \id_{B^0}}$};
\draw[->] (F)--(L) node [midway,right] {$\mathsmaller{i_{A^1}^X \times i_{A^0}^X}$};
\draw[->] (E)--(F) node [midway,above] {$\mathsmaller{\id_{A^1} \times u^-}$};
\draw[->] (L)--(S) node [midway,right] {$\mathsmaller{\Id_1^2(u^0)}$};
\draw[->] (A)--(S) node [midway,below] {$\mathsmaller{i_{B^1}^Y \times i_{B^0}^Y}$};

\end{tikzpicture}
\end{center}
\begin{align*}
 \Id_1^2(u^0)\big[\big(i_{A^1}^X \times i_{A^0}^X\big)\big(\id_{A^1} \times u^-\big)\big(a^1, b^0\big)\big]
 & = \Id_1^2(u^0)\big[\big(i_{A^1}^X(a^1), i_{A^0}^X(u^-(b^0))\big)\big]\\
 & = \big(u^0(i_{A^1}^X(a^1)), u^0(i_{A^0}^X(u^-(b^0)))\big)\\
 & = \big(i_{B^1}^Y(u^+(a^1)), i_{B^0}^Y(b^0)\big)\\
 & = \big[i_{B^1}^Y \times i_{B^0}^Y\big]\big(u^+(a^1), b^0)\big)\\
 & = \big[i_{B^1}^Y \times i_{B^0}^Y\big]\big(u^+ \times \id_{B^0}\big)(a^1, b^0).
\end{align*}
Clearly, $E^{\Pred^{\neq}}$ is injective on objects and arrows, hence $E^{\Pred^{\neq}}$ is an embedding. It is also full,
as if
$\big(u^0, u^+, u^-\big) \colon  \big(X; A^1, i_{A^1}^X \times i_{A^0}^X, A^0\big) \to 
 \big(Y; B^1, i_{B^1}^Y \times i_{B^0}^Y, B^0\big)$, then 
 $\big(u^0, u^+, u^-\big) \colon \big(X, \B A\big) \to \big(Y, \B B\big)$, because the commutativity of the 
last diagram implies the commutativity of the above two rectangles. 
\end{proof}

\section{The Chu construction and the antiparallel Grothendieck construction}
\label{sec: chugroth}

So far, we related the two constructions through the domain of the global Chu functor. The domain of the generalised global 
Chu functor has also some affinity to the Grothendieck construction. Next we discuss the relation between the 
two constructions themselves. A first result in this direction is the following result of Abramsky in~\cite{Ab18}, p.~14. 
Notice that instrumental to the proof of his result is a contravariant, or reverse, definition of the arrows in the
Grothendieck category. Namely, if $P \colon C^{\op} \to \CAT$, where $\CAT$ is the category of (large) categories,
an arrow $(f, \phi) \colon (a, x) \to (b, y)$ 
in the category $\Groth(\C C, P)$, where $x, y$ are objects in $P_0(a)$ and $P_0(b)$, respectively, is an arrow 
$f \colon b \to a$ in $\C C$ and an arrow $\phi \colon [P_1(f)]_0(x) \to y$ in $P_0(b)$. 
In the literature the standard approach to the definition of the category of elements or of the Grothendieck category is
is the covariant definition of the arrow $(f, \phi)$, where $f \colon a \to b$ and $\phi \colon x \to [P_1(f)]_0(y)$.
As we explain also later in this section, this reverse definition of the arrows in $\Groth(\C C, P)$ is necessary to 
Abramsky's result. Next follows the generalisation of Abramsky's result on an arbitrary $\CCC$.

\begin{proposition}[Abramsky 2018]\label{prp: abr}
 Let $\C C$ be a ccc and $\gamma \in C_0$. If $x \in C_0$, let $\Chu_x(\C C, \gamma)$ be the subcategory of 
 $Chu(\C C, \gamma)$ with objects triplets of the form $(a, f, x)$ and morphisms the pairs $(\phi^+, 1_x) \colon (a, f, x) 
 \to (b, g, x)$. If $h \colon x{'} \to x$, let the functor 
 $$h^* \colon \Chu_x(\C C, \gamma) \to \Chu_{x{'}}(\C C, \gamma),$$
 where $h^*_0(a, f, x) = \big(a, f \circ (1_a \times h), x{'}\big)$ and $h^*_1\big(\phi^+, 1_x\big) = \big(\phi^+, 1_x\big)$.
If $\Chu^{\gamma} \colon \C C^{\op} \to \CAT$ is the contravariant functor defined by 
 $$C_0 \ni x \mapsto \Chu_x(\C C, \gamma),$$
 $$\Chu^{\gamma}(h \colon x{'} \to x) = h^*,$$
 then the category $\Groth(\C C, \Chu_{\gamma})$ is the Chu category $\Chu(\C C, \gamma)$. 
\end{proposition}

\begin{proof}
See~\cite{Ga21}.
\end{proof}

The Chu construction can be seen as a special case of the antiparallel Grothendieck construction, or the antiparallel category of elements, on the product category, in case the $\CCC$ $\C C$ is locally small. In the next definition we could consider a product $\C C \times \C D$ instead of a product $\C C \times \C C$, and more options occur if larger products of categories are considered. If $\C C$ is a category, $a, b \in C_0$, and $S \colon (\C C \times \C C)^{\op} \to \Set$ a contravariant functor on $\C C \times \C C$, let the induced contravariant functors
$$S_a \colon \C C^{\op} \to \Set, \ \ S_a(c) = S_0(a, c) \ \ S_a(g \colon c \to c{'}) = S_1(1_a, g) \colon S_0(a, c{'}) \to S_0(a, c),$$
$$_bS \colon \C C^{\op} \to \Set, \ \ _bS(c) = S_0(c, b) \ \ _bS(g \colon c \to c{'}) = S_1(g, 1_b) \colon S_0(c{'}, b) \to S_0(c, b).$$

\begin{definition}\label{def: antiparallel}
Let $\C C$ be a category and $S \colon (\C C \times \C C)^{\op} \to \Set$. The $($contravariant$)$ antiparallel Grothendieck category $\Groth^{\leftrightarrows}\big(\C C \times \C C, S\big)$ has objects triplets $(a, x, u)$, where 
$a, x \in C_0$ and $u \in S_0(a, x)$, and morphisms pairs $\big(\phi^+, \phi^-\big) \colon (a, x, u) \to (b, y, v)$, where $\phi^+ \colon a \to b$ and $\phi^- \colon y \to x$ are morphisms in $\C C$ such that 
$[S_a(\phi^-)](u) = [_yS(\phi^+)](v)$ 
\begin{center}
\begin{tikzpicture}

\node (E) at (0,0) {$S_0(b, y)$};
\node[right=of E] (T) {};
\node[right=of T] (F) {$S_0(a, y)$.};
\node[above=of F] (A) {$S_0(a, x)$};

\draw[->] (E)--(F) node [midway,below] {$_yS(\phi^+)$};
\draw[->] (A)--(F) node [midway,right] {$S_a(\phi^-)$};

\end{tikzpicture}
\end{center}
If $\big(\theta^+, \theta^-\big) \colon (b, y, v) \to (c, z, w)$, let $\big(\theta^+, \theta^-\big) \circ \big(\phi^+, \phi^-\big)
= \big(\theta^= \circ \phi^+, \phi^- \theta^-\big)$.
Moreover, $1_{(a, x, u)} = (1_a, 1_x)$
\end{definition}

To justify the composition of morphisms in $\Groth^{\leftrightarrows}\big(\C C \times \C C, S\big)$, let the equalities:
\begin{equation}\label{eq: gr1}
S_1(1_a, \phi^-)](u) = [S_1(\phi^+, 1_y)](v)
\end{equation}
\begin{equation}\label{eq: gr2}
S_1(1_b, \theta^-)](v) = [S_1(\theta^+, 1_z)](w).
\end{equation}
We show the equality $S_1(1_a, \phi^- \circ \theta^-)](u) = [S_1(\theta^+ \circ \phi^+, 1_z)](w)$ as follows:
\begin{center}
\begin{tikzpicture}

\node (E) at (0,0) {$S_0(c, z)$};
\node[right=of E] (T) {};
\node[right=of T] (F) {$S_0(b, z)$};
\node[right=of F] (S) {};
\node[right=of S] (A) {$S_0(a, z)$.};
\node[above=of F] (B) {$S_0(b, y)$};
\node[above=of A] (C) {$S_0(a, y)$};
\node[above=of C] (D) {$S_0(a, x)$};

\draw[MyBlue,->] (E)--(F) node [midway,below] {$S_1(\theta^+, 1_z)$};
\draw[->] (F)--(A) node [midway,below] {$S_1(\phi^+, 1_z)$};
\draw[->] (B)--(F) node [midway,left] {$S_1(1_b, \theta^-)$};
\draw[->] (C)--(A) node [midway,right] {$S_1(1_a, \theta^-)$};
\draw[MyBlue,->] (B)--(C) node [midway,above] {$S_1(\phi^+, 1_y)$};
\draw[->] (D)--(C) node [midway,right] {$S_1(1_a, \phi^{-})$};

\end{tikzpicture}
\end{center}
\begin{align*}
[S_1(1_a, \phi^- \circ \theta^-)](u) & = \big[S_1\big((1_a, \phi^-) \circ (1_a, \theta^-)\big)\big](u) \\
& = \big[S_1(1_a, \theta^-) \circ S_1(1_a, \phi^-)\big](u) \\
& = [S_1(1_a, \theta^-)]\big(\big[S_1(1_a, \phi^-)\big](u)\big) \\
&  \stackrel{(\ref{eq: gr1})} = [S_1(1_a, \theta^-)]\big(\big[S_1(\phi^+, 1_y)\big](v)\big) \\
& = \big[S_1\big((\phi^+, 1_y) \circ (1_a, \theta^-)\big)\big](v) \\
& = \big[S_1(\phi^+ \circ 1_a, 1_y \circ \theta^-)\big](v) \\
& = \big[S_1(\phi^+, \theta^-)\big](v) \\
& = \big[S_1(1_b \circ \phi^+, \theta^- \circ 1_z)\big](v) \\
& = \big[S_1\big((1_b, \theta^-) \circ (\phi^+, 1_z)\big)\big](v) \\
& = [S_1(\phi^+, 1_z)]\big(\big[S_1(1_b, \theta^-)\big](v)\big) \\
& \stackrel{(\ref{eq: gr2})} = [S_1(\phi^+, 1_z)]\big(\big[S_1(\theta^+, 1_z)\big](w)\big) \\
& = \big[S_1\big((\theta^+, 1_z) \circ (\phi^+, 1_z)\big)\big](w) \\
& = [S_1(\theta^+ \circ \phi^+, 1_z)](w).
\end{align*}
The \textit{parallel} Grothendieck construction on $\C C \times \C C$ and $S$, with $\big(\phi^+, \phi^-\big) \colon (a, x, u) \to (b, y, v)$ is a pair of morphisms $\phi^+ \colon a \to b$ and $\phi^- \colon x \to y$ in $\C C$ is the standard category of elements over $\C C \times \C C$ and $S$. If $\C C$ is a locally small $\CCC$, we have the $\Set$-valued contravariant functor 
$$ \Hom(_{-} \times _{-}, \gamma)\big) \colon (\C C \times \C C)^{\op} \to \Set,$$
$$(a, b) \mapsto \Hom(a \times b, \gamma),$$
$$\Hom(_{-} \times _{-}, \gamma)\big)_1(\phi^+ \colon a \to a{'}, \phi^- \colon b \to b{'}) \colon \Hom(a{'} \times b{'}, \gamma)
\to \Hom(a \times b, \gamma),$$
$$\big[\Hom(_{-} \times _{-}, \gamma)\big)_1\big(\phi^+, \phi^-\big)\big](h) = h \circ \big(\phi^+ \times \phi^- \big)$$
\begin{center}
\begin{tikzpicture}

\node (E) at (0,0) {$\mathsmaller{a \times b} $};
\node[right=of E] (S) {};
\node[right=of S] (T) {$\mathsmaller{a{'} \times b{'}}$};
\node[right=of T] (U) {};
\node[right=of U] (F) {$\mathsmaller{\gamma}$.};

\draw[->] (E)--(T) node [midway,above] {$\mathsmaller{\phi^+ \times \phi^-}$};
\draw[->] (T)--(F) node [midway,above] {$h$};
\draw[MyBlue,->,bend right=30] (E) to node [midway,below] {$\mathsmaller{\big[\Hom(_{-} \times _{-}, \gamma)\big)_1\big(\phi^+, \phi^-\big)\big](h)}$} (F);
\end{tikzpicture}
\end{center}

\begin{proposition}\label{prp: chuspecialcase}
If $\C C$ is a locally small $\CCC$ and $\gamma \in C_0$, the Chu category $\Chu(\C C, \gamma)$ is the antiparallel Grothendieck category $\Groth^{\leftrightarrows}\big(\C C \times \C C, \Hom(_{-} \times _{-}, \gamma)\big)$.
\end{proposition}

\begin{proof}
In this case the defining equality~(\ref{eq: gr1}) takes the form 
$$\big[\Hom(_{-} \times _{-}, \gamma)\big)_1\big(1_a, \phi^-\big)\big](f) = 
\big[\Hom(_{-} \times _{-}, \gamma)\big)_1\big(\phi^-+, 1_y\big)\big](g)$$
i.e., $f \circ (1_a \times \phi^-) = g \circ (\phi^+ \times 1_y)$.
\end{proof}

In relation to Abramsky's result, and for a locally small $\CCC$ $\C C$ the previous result is maybe more interesting, as the functor $S$ is only $\Set$-valued, and not $\CAT$-valued. Next we describe the global version of the functor 
$\Hom(_{-} \times _{-}, \gamma)$.

\begin{proposition}\label{prp: sigmafunctor}
If $\C C$ is a locally small $\CCC$, the functor 
$$\Hom(_{-} \times _{-}, \ _{-}) \colon \C C \to \Fun\big((\C C \times \C C)^{\op}, \Set\big),$$
$$\big[\Hom(_{-} \times _{-}, \ _{-})\big]_0(\gamma) = \Hom(_{-} \times _{-}, \gamma),$$
$$\big[\Hom(_{-} \times _{-}, \ _{-})\big]_1( f \colon \gamma \to \gamma{'}) = \eta^f \colon \Hom(_{-} \times _{-}, \gamma) \To \Hom(_{-} \times _{-}, \gamma{'}),$$
$$\eta_{(a,b)}^f \colon \Hom(a \times b, \gamma) \to \Hom(a \times b, \gamma{'}),$$
$$\eta_{(a,b)}^f (h) = f \circ h$$
\begin{center}
\begin{tikzpicture}

\node (E) at (0,0) {$\mathsmaller{a \times b} $};
\node[right=of E] (T) {$\mathsmaller{\gamma}$};
\node[right=of T] (F) {$\mathsmaller{\gamma{'}}$.};

\draw[->] (E)--(T) node [midway,above] {$\mathsmaller{h}$};
\draw[->] (T)--(F) node [midway,above] {$\mathsmaller{f}$};
\draw[MyBlue,->,bend right=30] (E) to node [midway,below] {$\mathsmaller{eta_{(a,b)}^f (h)}$} (F);
\end{tikzpicture}
\end{center}
is an embedding. of $\C C$ into $\Fun\big((\C C \times \C C)^{\op}, \Set\big)$.
\end{proposition}

%
%
%
%
%
%
%
%
%
%
%
%
%
%
%
%
%
%
%
%
%
%
%
%
%
%
%

\vspace{8mm}

\noindent
\textbf{Acknowledgments}\\[1mm]
Our research was supported by LMUexcellent, funded by the Federal
Ministry of Education and Research (BMBF) and the Free State of Bavaria under the
Excellence Strategy of the Federal Government and the L\"ander.


\begin{thebibliography}{1}

\bibitem{Ab12} S.~Abramsky: Big toy models; Representing physical systems as Chu spaces, Synthese, 2012, 186:697-718.
\bibitem{Ab18} S.~Abramsky: Coalgebras, Chu Spaces, and Representations of Physical Systems, arXiv:01910.3959\\v1, 2009.
\bibitem{AR10} P.~Aczel, M.~Rathjen: \textit{Constructive Set Theory}, book draft, 2010.
\bibitem{Aw10} S.~Awodey: \textit{Category Theory}, Oxford University Press, 2010.
\bibitem{Ba79} M.~Barr: $^*$-\textit{Autonomous Categories}, LNM 752, Springer-Verlag, 1979.
\bibitem{Bi67} E.~Bishop: \textit{Foundations of Constructive Analysis}, McGraw-Hill, 1967.
\bibitem{BC72} E.~Bishop, H.~Cheng: \textit{Constructive Measure Theory}, Mem. Amer. Math. Soc. 116, 1972.
\bibitem{BB85} E.~Bishop, D.~S.~Bridges: \textit{Constructive Analysis}, Grundlehren der Math.~Wissenschaften 279,
Springer-Verlag, Heidelberg-Berlin-New York, 1985.
\bibitem{Br12} D.~S.~Bridges: Reflections on function spaces, Annals of Pure and Applied Logic 163, 2012, 101-110.
\bibitem{Ga21} L.~Gambarte: \textit{Chu categories}, Master Thesis, LMU, 2021, in preparation.
\bibitem{GJ20} H.~Geuvers, B.~Jacobs: Relating apartness and bisimulation, arXiv:2002.02512v1, 2020.
\bibitem{GJ60} L.~Gillman, M.~Jerison: \textit{Rings of Continuous Functions}, Van Nostrand, 1960.
\bibitem{GT07} E.~Giuli, W.~Tholen: A Topologists's View of Chu Spaces, Appl.~Categ.~Struct., 2007, 15:573-598.
\bibitem{He68} H.~Herrlich: \textit{Topologische Reflexionen und Coreflexionen}, LNM 78, Springer-Verlag, 1968.
\bibitem{Ja99} B.~Jacobs: \textit{Categorical Logic and Type Theory}, Elsevier Science B.V.~1999.
\bibitem{LS91} Y.~Lafont, T.~Streicher: Games semantics for linear logic, LICS, Washington, DC: 
IEEE Computer Society, 1991, 43-50.
\bibitem{MM92} S.~Mac Lane, I.~Moerdijk: \textit{Sheaves in Geometry and Logic}, Springer-Verlag, 1992.
\bibitem{Ma17} \textit{The Double Category of Paired Dialgebras on the Chu Category}, Master Thesis, Shahid Beheshti 
University, 2017. 
\bibitem{My75} J.~Myhill: Constructive Set Theory, J. Symbolic Logic 40, 1975, 347-382.
\bibitem{Pe15} I.~Petrakis: \textit{Constructive Topology of Bishop Spaces}, PhD Thesis, LMU Munich, 2015.
\bibitem{Pe19a} I.~Petrakis: Borel and Baire sets in Bishop Spaces, in F. Manea et. al. (Eds): 
Computing with Foresight and Industry, CiE 2019, LNCS 11558, Springer, 2019, 240--252. 
\bibitem{Pe19b} I.~Petrakis: Constructive uniformities of pseudometrics and Bishop topologies,
 Journal of Logic and Analysis, 11:FT2, 2019, 1-44.
\bibitem{Pe19d} I.~Petrakis: Dependent sums and Dependent Products in Bishop's Set Theory, 
in P.~Dybjer et. al. (Eds) TYPES 2018, LIPIcs, Vol. 130, Article No. 3, 2019.
\bibitem{Pe20} I.~Petrakis: \textit{Families of Sets in Bishop Set Theory}, Habilitationsschrift, LMU, Munich, 2020.
\bibitem{Pe20b} I.~Petrakis: Embeddings of Bishop spaces, Journal of Logic and Computation, 
 exaa015, 2020, https://doi.org/10.1093/logcom/exaa015.
 \bibitem{Pe20c} I. Petrakis: Functions of Baire class one over a Bishop topology, in M. Anselmo et al. (Eds.)
\textit{Beyond the Horizon of Computability}, CiE 2020, Springer, LNCS 12098, 2020, 215-227. 
\bibitem{Pe21a} I.~Petrakis: Direct spectra of Bishop spaces and their limits, Logical Methods in 
Computer Science, Volume 17, Issue 2, 2021, pp. 4:1-4:50. 
\bibitem{Pe21d} I. Petrakis: Closed subsets in Bishop topological groups, 2021,
https://arxiv.org/abs/2103.04718
\bibitem{Pe21e} I.~Petrakis: Bases of pseudocompact Bishop spaces, invited chapter in \textit{Handbook of Bishop Constructive Mathematics}, D. S. Bridges, H. Ishihara, M. Rathjen, H. Schwichtenberg (Eds.), Cambridge University Press, 2021, to appear. 
\bibitem{Pr99} V.~Pratt: \textit{Chu Space}, Notes for the School on Category Theory and Applications, University 
of Coimbra, 1999.
\bibitem{Ri16} E.~Riehl: \textit{Category Theory in Context}, Dover Publications Inc., 2016.
\bibitem{SW12} H.~Schwichtenberg, S.~Wainer: \textit{Proofs and Computations}, Cambridge University Press 2012.
\bibitem{Se89} R.~A.~G.~Seely: Linear logic, $^*$-autonomous categories and cofree coalgebras, in
\textit{Categories in computer science and logic. Contemporary Mathematics}, Vol.~92, Boston, MA: 
American Mathematical Society, 1989, 371-382.
\bibitem{Sh18} M.~Shulman: Linear Logic for Constructive Mathematics, arXiv:1805.07518v1, 2018.
\bibitem{Wa74} R.~C.~Walker: \textit{The Stone-\v{C}ech Compactification}, Springer-Verlag, 1974
\end{thebibliography}
\end{document}